\documentclass{article}

\usepackage[preprint]{neurips_2026}

\usepackage[utf8]{inputenc} % allow utf-8 input
\usepackage[T1]{fontenc}    % use 8-bit T1 fonts
\usepackage{hyperref}       % hyperlinks
\usepackage{url}            % simple URL typesetting
\usepackage{booktabs}       % professional-quality tables
\usepackage{amsfonts}       % blackboard math symbols
\usepackage{nicefrac}       % compact symbols for 1/2, etc.
\usepackage{microtype}      % microtypography
\usepackage{xcolor}         % colors

\usepackage{float}

\usepackage[normalem]{ulem}
\usepackage{framed}
\usepackage{tcolorbox}
\usepackage{pgfplots}
\usepackage[utf8]{inputenc}
\usepackage{booktabs}
\usepackage{amsmath,amsthm,verbatim,amssymb,amscd, graphicx}
\usepackage{natbib}
\usepackage{bbm}
\usepackage{verbatim}
\usepackage{xcolor}
\usepackage{hyperref}
\usepackage{amssymb,bm}
\usepackage{float}
\usepackage{mathrsfs} %Imports biblatex package

\usepackage{romannum}
\usepackage{graphicx}
\usepackage{subcaption}
\newtheorem{theorem}{Theorem}[section]

\newtheorem{remark}{Remark}[section]
\newtheorem{lemma}{Lemma}[section]
\newtheorem{proposition}{Proposition}[section]
\newtheorem*{definition}{Definition}

\newcommand{\pr}[1]{\mathbb{P}{\left(#1\right)}}

\newcommand{\wh}{\widehat}

\newcommand{\eps}{\varepsilon}
\newcommand{\mb}{\mathbb}
\def\l{\left}
\def\r{\right}
\newcommand\m{\mathcal}

\def\be#1{\begin{equation*}{#1}\end{equation*}}
\def\ben#1{\begin{equation}{#1}\end{equation}}

\newcommand{\E}{{\mathbb E}}

\newcommand{\simo}[1]{{\color{red}#1}}
\pgfplotsset{compat=1.18} 
\title{
Sub-Gaussian Confidence Intervals for Heavy-Tailed Data: Characterizing the Limits of Inference}

\author{
  Ilyes Hammouda \\
  Department of Statistics\\
  CREST ENSAE\\
  \texttt{hammouda@ensae.fr} \\
   \And
   Stanislav Minsker \\
   Department of Mathematics \\
   University of Southern California \\
   \texttt{minsker@usc.edu} \\
   \And
   Mohamed Ndaoud \\
    Information and Data Systems \\
   ESSEC Business School \\
   \texttt{ndaoud@essec.edu} \\
}

\begin{document}

\maketitle
\pagenumbering{arabic}

\begin{abstract}
We investigate the problem of uncertainty quantification in the mean estimation framework. Given an i.i.d. sample from a distribution with unknown mean $\mu$ and  variance $\sigma^2$, we want to construct an estimator $\wh\mu_n$ and provide a computable and size-optimal upper bound for the error $|\wh \mu_n - \mu|$. While estimators possessing strong deviation guarantees are well known, no data-dependent non-asymptotic upper bounds for their performance exist in general. We show that this gap can be bridged by introducing a parameter that characterizes the ``effective sample size'' available for uncertainty quantification. This parameter captures the difficulty of the problem, interpolating between ``easy'' cases where sub-Gaussian confidence intervals exist and ``hard" cases where their construction is impossible. Using this characterization, we design confidence intervals of optimal length that are fully adaptive to the unknown variance. Numerical experiments confirm that our approach maintains nominal coverage even in asymmetric and heavy-tailed regimes where other existing methods fail.
\end{abstract}

%###########################
\section{Introduction}
\label{sec:intro}
%###########################

 Estimation and uncertainty quantification problems related to the mean of a population are among the oldest and most studied in mathematical statistics, going back to the works of W. Gosset \citep{student1908probable}, \citet{fisher1925applications}, and  \citet{neyman1937outline}, among others. Over the past 15 years, particular attention has been drawn to the range of questions dealing with robust mean estimation -- a problem of mean estimation in the presence of contamination, often modeled via the outliers, or ``atypical'' observations. The research has largely centered around construction of optimal estimators while the problem of robust inference received less attention. 
 In this paper, we ask the following question: given a sequence of i.i.d. copies $X_1,\ldots,X_n$ of a random variable $X$ with mean $\mu$ and finite variance $\sigma^2$ where both parameters are unknown, when is it possible to construct a (non-asymptotic) confidence interval $\wh C_n(X_1,\ldots,X_n)$ of finite length and such that $\pr{\wh C_n \ni\mu}\geq 0.95$? Moreover, we are interested in characterizing the optimal size of such intervals. 
 
 Distinctions between the problems of rate-optimal adaptive estimation and adaptive inference are well known in the literature: there exist scenarios where the first task admits a solution while the second one does not. Classical examples include the problem of constructing adaptive confidence bands for the probability density functions of unknown smoothness \citep{low1997nonparametric} as well as the problem of sparse linear regression with unknown sparsity level \citep{nickl2013confidence}. 

 As the statistical research paradigm shifts from an asymptotic to a non-asymptotic perspective, new challenges in the basic problem of mean estimation have emerged. It has been shown by \cite{catoni2012challenging,devroye2016sub} that for distributions with finite variance\footnote{Samples from such distributions exhibit naturally occurring outliers, hence they present a natural model for contaminated data.} there exist sub-Gaussian mean estimators, that is, measurable functions $\wh\mu_n(\delta):=\wh \mu_n(X_1,\ldots,X_n)(\delta)$ such that 
 \[
 \pr{|\wh\mu_n(\delta) - \mu|\geq C\sigma\sqrt{\frac{\log(1/\delta)}{n}}} \leq \delta
 \]
 for all $e^{-cn}\leq \delta\leq e^{-1}$ where $c,C>0$ are absolute constants.  Moreover, it is impossible to improve this bound uniformly, aside from the absolute constants \citep{catoni2012challenging}. Note that $\wh\mu_n(\delta)$ does not require $\sigma$ to be known. When $X$ has normal distribution, the empirical mean $\Bar{X}_n$ satisfies the inequality
\ben{
\label{eq:t-distribution}
\mathbb{P}\left( |\Bar{X}_n - \mu| \geq C s_n \sqrt{\frac{\log(1/\delta)}{n}}\right) \leq \delta,
}
for $e^{-cn}\leq \delta \leq e^{-1}$ where $s^2_n = \frac{1}{n-1}\sum_{j=1}^n (X_j - \bar X_n)^2$ is the empirical variance (see Lemma \ref{lemma:gaussian} below for a version with explicit constants). The inequality \eqref{eq:t-distribution} can be equivalently rewritten as 
\[
\pr{\mu \in \l[ \bar X_n - Cs_n\sqrt{\frac{\log(1/\delta)}{n}},\bar X_n + Cs_n\sqrt{\frac{\log(1/\delta)}{n}}\r]}\geq 1-\delta.
\]
%for $\mu$ to obtain the standard confidence interval based on the quantiles of Student's t-distribution. 
Therefore, one may expect that, like in the mean estimation problem, this result for confidence intervals may be extended to all distributions in the domain of attraction of the normal law. 
However, such ``universal'' non-asymptotic confidence intervals do not exist: for any measurable $\wh L_n(X_1,\ldots,X_n), \ \wh U_n(X_1,\ldots,X_n)$ such that $-\infty<\wh L_n \leq \wh U_n<\infty$ almost surely, 
\[
\inf_{X\sim P\in \m P^{(2)}} \pr{[\wh L_n,\wh U_n]\ni \mu} = 0
\]
where $\m P^{(2)}$ is the class of all distributions with finite variance. This result follows from Corollary 3 in \citep{bahadur1956nonexistence}. 

The present paper provides further refinement of this impossibility result: we construct a nested sequence of subclasses of $\m P^{(2)}$ and characterize the optimal performance of confidence intervals over these subclasses. We also show that it is impossible to ``adapt'' to the scale defining these subclasses in a data-dependent way. Our ``effective sample size'' framework extends to distributions with finite
$(1+\beta)$-moments, for $\beta\in(0,1]$. The corresponding generalized
effective-mass parameter, optimal confidence-interval rate, and matching lower
bound are stated in Section \ref{sec:existence} and proved in Appendix \ref{section:proofs_upper}. 

Let us mention that another classical framework for modeling outliers is the Huber's contamination model \citep{huber1964robust} where the distribution of $X$ is a mixture $(1-\eps)P + \eps Q$ of a distribution $P$ with mean $\mu$ (a parameter of interest), and an arbitrary distribution $Q$. 
When $\eps$ is known and $P$ has variance bounded by a known constant $\sigma_0^2$, \citet{wang2023huber} showed that it is possible to construct confidence intervals of length scaling as $\sigma_0\sqrt{\eps}$ (also see \citep{bateni2020confidence,maurer2009empirical,shekhar2023near} for results related to a class of compactly supported distributions). However, when $\eps$ is unknown, the price for adaptation is very high: recently, \citet{luo2024adaptive} showed that if $P$ is the normal distribution with variance $\sigma^2$, then the confidence interval has to have length at least $\sigma(\log(n))^{-1/2}$ even if there are no outliers in the data ($\eps=0$) but potential contamination is suspected. If moreover $\sigma^2$ is unknown, finite-length confidence intervals do not exist (even though the ``inliers'' are normally distributed).

%\begin{notation}
%#####################
\subsection{Notation}
\label{section:notation}
%#####################

We use $c,c_1,C,C_1,\alpha$, etc. to denote absolute positive constants and $c_n,\alpha_n$, etc. to denote constants depending on $n$. For two sequences $(a_n)_n$ and $(b_n)_n$ we say that $a_n = \Omega(b_n)$ (resp. $a_n = \mathcal{O}(b_n)$) if $a_n \geq C b_n$ (resp. $a_n \leq C b_n$) for some $C>0$ and large values of $n$. We say $a_n \simeq b_n$ if $a_n = \Omega(b_n)$ and $a_n = \mathcal{O}(b_n)$.  We also say that $a_n = o(b_n)$ iff $a_n/b_n\to 0$. 
For a positive integer \( N \), \([N]\) will denote the set \(\{1, \dots, N\}\).
For a given set $A$, $|A|$ will denote the cardinality of the latter. For a sequence of $K$ elements $(u_n)_{n\in [K]}$, $\text{med}(u_1 \cdots u_K)$ will refer to the median of the sequence, i.e., if we order the $K$ elements in ascending order, such that $u_{(1)} \leq u_{(2)} \cdots u_{(K)} $, then 
\begin{equation*}
    \text{med}(u_1, \dots, u_K) = 
\begin{cases}
u_{\left(\frac{K+1}{2}\right)}, & \text{if } K \text{ is odd} \\
\frac{1}{2}\left[u_{\left(\frac{K}{2}\right)} + u_{\left(\frac{K}{2}+1\right)}\right], & \text{if } K \text{ is even}
\end{cases}.
\end{equation*}
We will also use notations $\left(\xrightarrow{d} \right)$ to denote convergence in distribution, $\left(\xrightarrow{a.s.}\right)$ for almost sure convergence and $\left(\xrightarrow{\mathbf{P}} \right)$ for convergence in probability. Finally, given a random variable $\xi$ and a probability distribution $P$, we write $\xi\sim P$ to say that $\xi$ is distributed according to $P$. Additional notation will be introduced on demand. 
%\end{notation}

%###############################
\section{When do confidence intervals exist?}
\label{sec:existence}
%###############################

While there are no non-trivial adaptive confidence intervals for the class $\m P^{(2)}$, they do exist for large sub-classes $\tilde{\m P}$ of $\m P^{(2)}$. A notable construction of such intervals is due to \citet[][Lemma 6.1]{devroye2016sub} where it is shown that it suffices to require that the sample averages have ``approximately symmetric'' distributions for all sufficiently large samples. 
%then required confidence intervals exist. 
Specifically, the sufficient condition is that $\forall j\geq \lfloor\frac{n}{\log(1/\delta)} \rfloor$ and all $P\in \tilde{\m P}$,
\ben{
\label{assumption:1}
\min(p_-(j),p_+(j))\geq 1/3,
}
where $p_-(j) = \pr{\frac{1}{j}\sum_{i=1}^j X_i\leq \mu}$ and 
$p_+(j) = \pr{\frac{1}{j}\sum_{i=1}^j X_i\geq \mu}$\footnote{The constant $1/3$ in \eqref{assumption:1} can be replaced by any $0<C<1/2$.}. Of course, assumption \eqref{assumption:1} holds for every fixed $P\in \m P^{(2)}$ and all $j\geq j(P)$ whenever $j(P)$ is sufficiently large in view of the Central Limit Theorem, but it does not hold uniformly over $\m P^{(2)}$ since $\sup_{P\in \m P^{(2)}} j(P)=\infty$. Natural classes of distributions $\tilde{\m P}$ that satisfy assumption \eqref{assumption:1} are the ones that admit uniform bounds in the central limit theorem, for example, such that for some $p>2$, 
\[
\sup_{P\in \tilde{\m P}} \frac{\mathbf{E}^{1/p}|X-\mu|^p}{\sigma} \leq C\l( \tilde{\m P}\r)<\infty.
\]
In this case, condition \eqref{assumption:1} is implied by Berry-Esseen type inequalities \citep{petrov2012sums}. 

\textbf{Main contribution.} In this paper, we construct a family $\{ \m P_\alpha(n,\delta) \}_{0\leq \alpha\leq 1}$ of nested subclasses of $\m P^{(2)}$, formally defined below, such that
\[
\m P^{(2)}=\m P_0(n,\delta)
\supseteq \m P_{\alpha_1}(n,\delta)
\supseteq \m P_{\alpha_2}(n,\delta),
\qquad \alpha_1\leq\alpha_2,
\]
and such that for each $\m P_\alpha(n,\delta)$ there exist optimal-length sub-Gaussian confidence intervals.

\begin{definition}
Let $n$ be an integer, $0<\delta<1$. Let $X$ be a random variable with mean $\mu$ and variance $\sigma^2>0$, and set $\xi:=\xi(X)=\frac{X-\mu}{\sigma}$. Define
\[
\alpha_{n,\delta}(\xi) := {\sup} \left\{\gamma\geq 0 \;, \mathbf{E}\l(\xi^2 \mathbf{1}\l(\xi^2 \leq \frac{n\gamma}{\log(1/\delta)}\r)\r) \geq  \gamma \right\},
\]
and the associated class $\mathcal{P}_{\alpha}(n,\delta) :=\left\{ P \in \m P^{(2)}: \ X\sim P \text{ and } \alpha_{n,\delta}(\xi) \geq \alpha \right\}$. 
\end{definition}
%\stas{(replaced $\xi(X)\sim P$ by $X\sim P$, otherwise we only include distributions with variance 1)} \simo{I think it s best if we define the noise as standard in that sense we are adaptive to the noise distribution $P$ and the noise level $\sigma$}
\noindent Notice that $\alpha_{n,\delta}(\xi)$ is always finite. Indeed, $\gamma=0$ belongs to the set 
\[
\m A = \left\{\gamma\geq 0 \;, \mathbf{E}\l(\xi^2 \mathbf{1}\l(\xi^2 \leq \frac{n\gamma}{\log(1/\delta)}\r)\r) \geq  \gamma \right\},
\]
and moreover $\m A\subseteq [0,1]$. Hence, $0\leq \alpha_{n,\delta}(\xi) \leq 1$. Observe also that $\alpha_{n,\delta}(\xi) \neq 0$ implies that $n\geq \log(1/\delta)$, a known requirement necessary (up to absolute constants) for existence of sub-Gaussian mean estimators. A lower bound on $\alpha$ excludes distributions with a ``dominant'' mass at $0$ from the class $\mathcal{P}_{\alpha}(n,\delta)$, as such distributions often lead to various impossibility results. For example, it is easy to see that the inequality 
\[
\mathbb{P}(1/2 \leq \xi^2\leq 4) \geq 1/2
\]
implies that $\alpha_{n,\delta}(\xi) \geq 1/4$ (of course, the constants in this relation can be modified). As we show below, $\alpha_{n,\delta}(\xi)$ quantifies the proportion of a sample of size $n$ that can be used to estimate $\sigma^2$, and hence we can think of the product $n\alpha_{n,\delta}(\xi)$ as the \emph{``effective sample size''}. 

Let us also mention that for any $\eps>0$, there exists $\alpha(\eps)$ small enough and a random variable $X$ such that its law $P_X \in \mathcal{P}_{\alpha}$ and
    $\min(p_-(j),p_+(j)) \leq \eps$, as shown formally in Lemma \ref{lem:counter}. Therefore, the approach of \citet{devroye2016sub} is limited to the case $\alpha_{n,\delta}$ is a constant bounded away from $0$.
    We are ready to state our main result. 
 % Prior impossibility results are usually based on distributions with a dominant mass on $0$. Our condition $\alpha_{n,\delta}(\xi)$ prevents such cases. One may also notice that $\mathbf{P}(1/2 \leq \xi^2\leq 4) \geq 1/2$ implies that $\alpha_{n,\delta}(\xi) \geq 1/4$. To some extent, $\alpha_n$ quantifies the available mass of observations that will help estimate $\sigma$, and hence we should think of $n\alpha_{n,\delta}(\xi)$ as the effective sample size.

\begin{tcolorbox}
\begin{theorem}
\label{theorem:main}
\textbf{(i)} Let $0<\delta<1$ and assume that $\xi(X)$ is distributed according to $P\in \m P^{(2)}$. There exists a sequence of confidence intervals $\wh C_n(X_1,\ldots,X_n) = [\wh L_n(X_1,\ldots,X_n),\ \wh U_n(X_1,\ldots,X_n]$ such that for all $\alpha_n\in (0,1]$,  
%\stas{(do we need $\inf_\sigma$ below? also, should we say how $\alpha_n$ and $\alpha$ are related?)}\simo{in my opinion the inf is relaetd to whether we want to talk about adaptation to $\sigma$ and $P$ separately. Also $\alpha_n$ and $\alpha$ are not related. They are completely independent quantities.}
\[
%\inf\limits_\sigma
\inf_{P\in \m P_{\alpha_n}(n,\delta)} \pr{\wh C_n\ni\mu}\geq 1-\delta \quad\text{  and }
\]
\[
%\inf_{P\in \m P_{\alpha_n}(n,\delta)}
\inf_{P\in \m P^{(2)},\sigma=\sigma_\ast}\pr{|\wh C_n|\leq K\sigma_\ast\sqrt{\frac{\log(1/\delta)}{n\alpha_n}}} \geq 1 -\delta,
\]
%\stas{(can the inf in the second statement above be over $\m P^{(2)}$ then?}
as long as $n\alpha_n \geq C \log(1/\delta)$, where $|\wh C_n| = |\wh U_n - \wh L_n|$ and $C,K>0$ are absolute constants. Moreover, $\wh L_n,\wh U_n$ are adaptive to $P$ 
%\stas{(does adaptivity to $P$ automatically mean adaptivity to $\sigma$?)} 
and $\alpha_{n,\delta}(\xi)$, meaning that they do not require any of these parameters to be known in advance, except for the lower bound $\alpha_n$ specifying the class $\m P_{\alpha_n}(n,\delta)$.

%Can we claim that the length of the confidence interval is actually controlled by $\alpha_{n,\delta}(\xi)$ and not its lower bound $\alpha_n$?

\textbf{(ii)} Let $\alpha_n,\gamma_n,\delta\in (0,1)$ be such that $\gamma_n>\alpha_n$, and let $c>0$ be a sufficiently small absolute constant. Then there is no sequence of adaptive  confidence intervals $\widetilde C_n(X_1,\ldots,X_n)$  that  simultaneously satisfy the conditions 
%\stas{(I think inf over $\sigma$ is not necessary here as well, it is already implicitly included in inf over P)}
\[
%\inf\limits_\sigma
\inf\limits_{P\in \m P_{\alpha_n}(n,\delta)}\pr{\widetilde C_n\ni \mu}\geq 1-\delta \quad \text{ and }
\]
%for all $X\sim P\in \m P_{\alpha_n}(n,\delta)$,
\[
\inf_{P\in \m P_{\gamma_n}(n,\delta),\sigma=\sigma_\ast} \pr{|\widetilde C_n| \leq c\sigma_\ast \sqrt{\frac{\log(1/\delta)}{n\alpha_n}} }\geq 1-\delta.
\]
%\stas{(is $\inf$ over $\m P_{\alpha_n}$?)}\simo{the second inf is over $\alpha$. In the sense that no matter what $\alpha$ you $P_1$ distribution has (Type 2 error) you will always pay the price $\alpha_n$ in your confidence interval. I suggest to replace $\alpha$ by $\gamma$ or other constant here}.\stas{(is it essentially equivalent if we say that adaptation to $\alpha_n$ is not possible?)}
%for some small constant $c>0$.
\end{theorem}
\end{tcolorbox}
\begin{remark}
\label{remark:1}
\begin{enumerate}
\item[(a)] This theorem does not contradict the Central Limit Theorem as for every fixed distribution $P\in \m P^{(2)}$, and $X\sim P$, $\alpha_{n,\delta}(\xi(X)) \to 1$ when $n\to \infty$. 
    %For the CLT we first fix the distribution then take $n \to \infty$. 
\item[(b)] Theorem \ref{theorem:main} establishes that the optimal length for adaptive confidence intervals over the class $\m P_{\alpha_n}(n,\delta)$ is of order $\sigma \sqrt{\frac{\log(1/\delta)}{n\alpha_n}}$. 
\item[(c)] While part (i) guarantees existence of such confidence intervals, part (ii) states that it is essentially impossible to adapt to the value of $\alpha_n$: even if the data come from a distribution belonging to $\m P_{\gamma_n}(n,\delta)$ with $\gamma_n > \alpha_n$, any confidence interval that guarantees uniform coverage of order $\delta$ over the class $\m P_{\alpha_n}(n,\delta)$ has to have length at least $\tilde C\sigma \sqrt{\frac{\log(1/\delta)}{n\alpha_n}}$. 
\item[(d)] Theorem \ref{theorem:main} characterizes the interpolation between the ``easy" regime, where $\alpha_n$ is bounded away from zero by a constant, and the ``hard" regime, where $\alpha_n \to 0$. Here, the product $n\alpha_n$ serves as the effective sample size, as the length of the adaptive sub-Gaussian confidence interval scales precisely with $(n\alpha_n)^{-1/2}$.
\end{enumerate}
\end{remark}
%In other words, in order to achieve uniform coverage of level $\delta$ over the class $P_{\alpha_n}(n,\delta)$, the length of confidence intervals needs to be of order $\sigma \sqrt{\frac{\log(1/\delta)}{n\alpha_n}}$ even when $\alpha_{n,\delta}(\xi)$ is known. 

\textbf{Interpretation of the hard regime.} The case when \(\alpha_{n,\delta}\) is close to $0$ should not be interpreted as simply an artificial theoretical construct. Rather, it describes distributions with a ``variance gap'': observations of small or moderate magnitude account for only a small fraction of the population variance, while observations with high values have very low
probability of occurrence but still contribute most to the variance. Such a regime naturally arises in imbalanced mixture distributions. For instance, in the construction underlying our lower bounds, the minority component may have weight of order \(\log(1/\delta)/n\). For \(n=200\) and \(\delta=0.05\), this corresponds to about \(1.5\%\) of the population, or roughly three observations on average. These observations may look like outliers relative to the dominant component, but they belong to the target distribution and removing them changes the estimand. Thus, a typical sample may appear approximately Gaussian with small variance, while a few legitimate observations determine both its centering and a substantial fraction of its population variance. There are natural examples where such rare ``outliers''
carry most of the useful information in the sample: financial returns with occasional large losses or gains, insurance portfolios with extremely rare catastrophic claims that can lead to insurer's
financial losses if left unaccounted for, errors in measurement systems such as GPS positioning where one wants to account for the possible loss of satellite connection.

Note that within the class of distributions where $\alpha_n \asymp n^{-1} \log(1/\delta)$, all optimal adaptive intervals have length of order $\sigma$. This implies that in such cases, one cannot expect the confidence intervals to shrink around the true mean as the sample size $n$ increases.

A deviation inequality whose radius depends on the unknown scale \(\sigma\) does not by itself provide a computable confidence interval. The construction of \cite{devroye2016sub} requires the approximate-symmetry condition \eqref{assumption:1}, whereas Lemma \ref{lem:counter} shows that this condition may fail within \(\mathcal P_{\alpha_n}(n,\delta)\) when \(\alpha_n\) is small.

%In practice, $\alpha_n$ should be chosen to ensure that the interval length remains manageable. \stas{(I think this comment invites ``easy'' critique from the referees; perhaps, in practice we always choose $\alpha$ to be a constant)} 
Provided that $n\alpha_n = \Omega\left( \log(1/\delta)\right)$, 
%\stas{(do we mean $n\alpha_n\geq C\log(1/\delta)$? because $\alpha_n=0$ satisfies $n\alpha_n = \mathcal{O}\left( \log(1/\delta)\right)$)} \simo{Yes} 
the resulting interval is guaranteed to maintain the prescribed coverage. In table \ref{table:summary_results} below, we include a summary of existing results showing how our findings interpolate findings from the prior work.

\begin{table}[!ht]
\centering
\caption{Comparison of optimal adaptive confidence interval lengths for the mean $\mu$ under two moments.}
\label{table:summary_results}
\begin{tabular}{@{}llll@{}}
\toprule
\textbf{Case} &  \textbf{Optimal Interval Length} & \textbf{Reference} \\ \midrule
$\sigma$ known &  $\sigma \sqrt{\frac{\log(1/\delta)}{n}}$ &   \citep{nemirovskij1983problem} \\ 
\addlinespace
$\sigma$ unknown + symmetry condition  & $\sigma \sqrt{\frac{\log(1/\delta)}{n}}$ & \citet{devroye2016sub} \\
\addlinespace
$\sigma$ unknown and $P_X \in \mathcal{P}_{\alpha_n}$
& $\sigma \sqrt{\frac{\log(1/\delta)}{n\alpha_n}}$ & \textbf{Theorem \ref{theorem:main}} \\
\addlinespace
$\sigma$ unknown & $\infty$ (Does not exist) & \citet{bahadur1956nonexistence} \\
  \bottomrule
\end{tabular}
\end{table}
\paragraph{Extension to lower moments.}
Our results extend beyond the finite-variance setting. Assume that
$\mathbf{E}|X-\mu|^{1+\beta}<\infty$ for some $\beta\in(0,1]$, and let
$
v_{1+\beta}
    := \bigl(\mathbf{E}|X-\mu|^{1+\beta}\bigr)^{1/(1+\beta)}.
$
Defining the corresponding effective-mass parameter
$\alpha^{(\beta)}_{n,\delta}$ through the truncated $(1+\beta)$-moment,
the optimal confidence interval length becomes
\[
v_{1+\beta}
\left(
    \frac{\log(1/\delta)}
         {n\alpha^{(\beta)}_{n,\delta}}
\right)^{\frac{\beta}{1+\beta}}.
\]
Both the upper bound and the matching lower bound extend to this setting:
the proof follows the same scheme after replacing the truncated
second-moment condition and the variance estimator by their
$(1+\beta)$-moment analogues. The case $\beta=1$ recovers the
finite-variance results developed throughout the paper. We provide the
formal definition of $\alpha^{(\beta)}_{n,\delta}$ and the corresponding
proof details in Appendix \ref{section:proofs_upper}.

\textbf{Roadmap of the proof of Theorem \ref{theorem:main}.} We now outline the key steps and ideas of the proof.
\begin{enumerate}
    \item \textbf{Construct a preliminary robust mean estimator.}
    We start from a sub-Gaussian estimator $\widehat{\mu}$ satisfying
    \[
    |\widehat{\mu}-\mu|
    \lesssim
    \sigma\sqrt{\frac{\log(1/\delta)}{n}}
    \]
    under the finite variance assumption. The same observations are then reused for scale estimation. \emph{No sample splitting is required}.

    \item \textbf{Estimate the scale through blockwise truncated mass.}
    We partition the sample into $K\asymp\log(1/\delta)$ blocks and compute blockwise empirical second moments around $\widehat{\mu}$. The definition of $\alpha_{n,\delta}$ guarantees that each block contains enough moderate observations to provide a lower bound of order $\alpha_n\sigma^2$, while the finite variance assumption gives a constant-order upper bound. Median aggregation amplifies these fixed-probability blockwise controls, yielding Proposition \ref{prop:sigma_estimator}:
    $
    c\,\alpha_n\sigma^2
    \leq
    \widehat{\sigma}_{\mathrm{MoM}}^2
    \leq
    C\,\sigma^2 \;
    \text{with probability at least }1-\delta.
    $

    \item \textbf{Build the confidence interval.}
    Combining the robust mean estimator with the scale estimate gives an interval with radius
    $
    \widehat{\sigma}_{\mathrm{MoM}}
    \sqrt{\frac{\log(1/\delta)}{n\alpha_n}},
    $
    and hence length of order
    $
    \sigma\sqrt{\frac{\log(1/\delta)}{n\alpha_n}}.
    $ 
    This proves the upper bound in Theorem \ref{theorem:main}.

    \item \textbf{Design an $\alpha_n$-sensitive lower bound pair.}
    The lower bound requires more than a classical two-mass-point construction. We must simultaneously ensure that the null belongs to the ``difficult'' class $\mathcal{P}_{\alpha_n}(n,\delta)$, that the alternative belongs to an ``easier'' class $\mathcal{P}_{\gamma_n}(n,\delta)$, and that their means differ at the target scale
    $
    \sigma_\ast
    \sqrt{\frac{\log(1/\delta)}{n\alpha_n}}.
    $ 
    Writing $L=\log(1/\delta)$, we choose
    \[
    \tau=\frac{L}{n},
    \qquad
    A\asymp\sqrt{\frac{L}{n\alpha_n}},
    \]
    and construct two mixtures sharing the same dominant Gaussian component:
    \[
    P_0=(1-\tau)\mathcal{N}(A,1)+\tau\delta_{z_0},
    \qquad
    P_1=(1-\tau)\mathcal{N}(A,1)+\tau\delta_0,
    \]
    where $z_0$ is chosen so that $P_0$ is centered. The rare component under $P_0$ is genuine and population-relevant: it centers the distribution and carries the variance needed to place the standardized law in the hard class. Replacing this component by an atom at zero produces an easier alternative whose mean is shifted at the desired scale.

    \item \textbf{Exploit the common-mixture event.}
    If none of the $n$ observations come from the rare component, the two samples are generated from exactly the same Gaussian distribution. This event has probability
    \[
    (1-\tau)^n
    =
    \left(1-\frac{L}{n}\right)^n
    \gtrsim e^{-L}
    =
    \delta.
    \]
    Therefore, the product measures retain an overlap of order $\delta$, and every test must incur total error at least of this order. A Neyman--Pearson argument yields the separation lower bound of order 
    $
    \rho_n
    \gtrsim
    \sigma_\ast
    \sqrt{\frac{\log(1/\delta)}{n\alpha_n}}.
    $ 

    \item \textbf{Transfer testing limits to confidence intervals.}
    Test/confidence set duality converts the testing lower bound into the optimal length and non-adaptivity statements of Theorem \ref{theorem:main}. The same mixture principle also gives Proposition \ref{prop:sigma_lower} by producing nearly indistinguishable models whose variances differ by a factor of order $\alpha_n$.

\end{enumerate}

% The technical novelty of the lower bound is therefore the carefully calibrated mixture: its rare component simultaneously controls the population mean, the scale, membership in the effective-sample-size classes, and the overlap of the product measures. These requirements are not captured by a standard two-mass-point impossibility argument.

\section{A hypothesis testing perspective}\label{sec:test}
 
%Theorems 3.2 and 4.2 in \cite{devroye2016sub} imply that sub-Gaussian confidence intervals do not exist if only finite variance is assumed. 
%Given $n$ i.i.d. observations of $X$ a random variable with finite variance, the goal is to construct confidence intervals for the mean $\mathbf{E}(X)$. 
%If $X$ has only two finite moments we can construct robust estimators similar to Median-Of-Means that achieve a optimal deviation bounds for fixed level of confidence $\delta$ under the condition $n =\Omega(\log(1/\delta))$. The challenge remains to estimate the std $\sigma$ under only two moments. 
%\textbf{Question:} Under what conditions on $n$, the distribution of $X$ and $\delta$  can we construct adaptive and sub-Gaussian confidence intervals? 
%\stas{Another question: for which $\delta$ is it possible to construct a confidence interval with finite length? (not necessarily sub-Gaussian)}

In order to get a better understanding of the problem related to adaptive confidence intervals, we will reduce it to a hypothesis testing problem. We still assume that $X = \mu + \sigma \xi$, where $\mu \in \mathbb{R},\sigma >0$ and $\xi$ a centered random variable such that $\mathbf{E}(\xi^2) = 1$. 
%\stas{Let $\m P^{(2)}(\sigma_*)\subset \m P^{(2)}$ consist of all distributions with variance equal to $\sigma_*^2>0$}. 
We will also denote via $P_\xi \in \m P^{(2)}$ the distribution of $\xi$. 
Consider the following testing problem:
\[
\mathbf{H}_0: \mu = 0 \quad \text{vs} \quad \mathbf{H}_1: |\mu| \geq \rho_n,
\]
where $\rho_n>0$. Given $n$ i.i.d. observations $X_1,\ldots,X_n$, we define the minimax risk $\mathcal{R}^*(n,\rho_n,\sigma_*)$ via
\[
\mathcal{R}^*(n,\rho_n,\sigma_*):= \underset{\hat{T}}{\inf}\left( \underset{\mu=0, \sigma=\sigma_\ast,P_\xi \in \m P^{(2)}}{\sup}\pr{\hat{T} = 1} + \underset{|\mu|\geq \rho_n, \sigma=\sigma_\ast, P_\xi\in \m P^{(2)}}{\sup}\pr{\hat{T} = 0}\right),
\]
where the infimum is taken over all measurable test functions $\hat{T}(X_1,\dots,X_n)$. We are interested in necessary and sufficient ``separation threshold'' $\rho_n^*(\delta)$ that yields minimax risk of order $\delta$.

Under just two finite moments, it is well known that in order to achieve minimax risk of order $\delta$, the necessary and sufficient separation $\rho^*_n(\delta)$ is of order $\rho_n^*(\delta):=\sigma_* \sqrt{\log(1/\delta)/n}$ as long as $n=\Omega(\log(1/\delta))$. An optimal testing procedure is given by 
\[
\hat{T} = \mathbf{1}\{|\mathrm{MoM}(X_1,\dots,X_n)| \geq 4\sigma_*\sqrt{\log(1/\delta)/n}\},
\]
where $\mathrm{MoM}(X_1,\dots,X_n)$ is the median-of-means estimator \citep{nemirovskij1983problem}. \emph{This fact shows that testing the mean under two finite moments is not more difficult than testing under Gaussian noise}. Unfortunately, this testing procedure requires prior knowledge of $\sigma_*$. Our next goal is to understand the required separation for testing when $\sigma_*$ is unknown.

In order to characterize adaptive testing, we need to define the risk differently. The following definition of the minimax risk is inspired by the work \cite{luo2024adaptive}. Specifically, let
\[
\mathcal{R}(n,\rho_n,\sigma_*):=
 \underset{\hat{T}}{\inf}\left( \underset{\mu=0,\sigma>0,P_\xi \in \m P^{(2)}}{\sup}\pr{\hat{T} = 1} + \underset{|\mu|\geq \rho_n, \sigma=\sigma_\ast, P_\xi\in \m P^{(2)}}{\sup}\pr{\hat{T} = 0}\right),
\]
Here, the Type \Romannum{1} error is uniform over $\sigma>0$ in order to allow for adaptation, while the Type \Romannum{2} error is given for a fixed value $\sigma=\sigma_*$. This definition of the risk allows us to capture the optimal adaptive separation $\rho_n$ as a function of $\sigma_*$. 

The asymmetric treatment of the scale in this definition is intentional. Type-I error is controlled uniformly over \(\sigma>0\), since validity must hold without prior knowledge of the scale, whereas Type-II error is evaluated at a fixed \(\sigma=\sigma_\ast\), since the target interval length itself depends on the true scale. Up to immaterial constants in \(\delta\), this corresponds to requiring a level \(\delta\) test with power at least \(1-\delta\), and is precisely the testing criterion obtained from the desired properties of adaptive confidence intervals through test/confidence-set duality.

\begin{proposition}\label{prop:infty}
    For all values of $\rho_n\geq 0$, we have $\mathcal{R}(n,\rho_n,\sigma_*) = 1$.
\end{proposition}

Proposition \ref{prop:infty} implies that we can't test the mean when $\sigma$ is unknown, i.e. $\rho_n^*(\delta)=\infty$.
The proof of this result is based on a simple construction of two distributions both of which have a large ``mass'' at $0$ so that they cannot be distinguished based on a sample of size $n$. This ``all-or-nothing'' risk reflects the impossibility of testing when $\sigma$ is entirely unknown. However, by restricting to the class $\mathcal{P}_\alpha$ we recover meaningful results. Indeed, if we define for each $\delta$ and $\gamma_n>\alpha_n$
\begin{multline*}
\mathcal{R}(n,\rho_n,\sigma_*,\alpha_n,\gamma_n,\delta)
\\
:=\underset{\hat{T}}{\inf}\left( \underset{\mu=0,\sigma >0,P_\xi\in \mathcal{P}_{\alpha_n}(n,\delta)}{\sup}\pr{\hat{T} = 1} + \underset{|\mu|\geq \rho_n,\sigma=\sigma_\ast,P_\xi \in \mathcal{P}_{\gamma_n}(n,\delta)}{\sup}\pr{\hat{T} = 0}\right),
\end{multline*}
then we get the following statement.
%\stas{[Can we replace $\m P_\gamma$ by $\m P^{(2)}$?]}
\begin{proposition}\label{prop:alpha}
    For each value of $\delta$ such that $n\alpha_n = \Omega(\log(1/\delta))$ and $\gamma_n>\alpha_n$, we have that 
    \[
    \mathcal{R}(n,\rho_n,\sigma_*,\alpha_n,\gamma_n,\delta) \leq \delta \text{ iff } \rho_n =\Omega\left( \sigma_* \sqrt{\frac{\log(1/\delta)}{\alpha_n n}}\right ).
    \]
\end{proposition}
In that case we show (via explicitly constructing a test function) that in order to achieve a testing error of order $\delta$, the optimal separation must be of order $\rho_n^*:= \sigma_*\sqrt{\frac{\log(1/\delta)}{n \alpha_n}}$. Note that the bound is uniform over $\sigma>0$. 
Observe also that we recover the impossibility result as $n\alpha_n \underset{n\to \infty}{\to} 0$. In other words, in order to get shrinking confidence intervals, we need to ensure that  $n\alpha_n \underset{n\to \infty}{\to} \infty$. 
Finally, to understand the role of parameter $\gamma_n$, we refer the reader to the previously stated Remark \ref{remark:1} (c).

Duality between existence of tests and confidence sets immediately implies the following lemma.
\begin{lemma}
    For any $\rho_n$ such that that $\mathcal{R}(n,\rho_n,\sigma_*,\alpha_n,\gamma_n,\delta) \leq \delta$, there exists a confidence interval with coverage probability at least $1-\delta$ and length at most $\rho_n$, also with probability at least $1-\delta$. Alternatively, if we can find an adaptive confidence interval of level $\delta$ such that its length is uniformly less than $\rho_n$ with probability at least $1-\delta$, then $\mathcal{R}(n,\rho_n,\sigma_*,\alpha_n,\gamma_n,\delta) \leq \delta$.
\end{lemma}

\section{On the construction of optimal adaptive confidence intervals}
For this section, we assume that $n\alpha_n = \Omega(\log(1/\delta))$. Notably, this condition is not significantly more restrictive than $n = \Omega(\log(1/\delta))$, as emphasized earlier, since $n\alpha_n$ can be interpreted as the ``effective sample size" derived from the data. 

\subsection{An estimator of the variance}

To construct an estimator of the variance, we need an initial estimator of the mean. To this end we choose an estimator $\hat{\mu}$ of $\mu$ such that $|\hat{\mu}- \mu| \leq 10 \sigma \sqrt{\frac{\log(1/\delta)}{n}}$ with probability $1-\delta$. While there are several viable candidates in the literature, we choose the most well-known and arguably the simplest median of means (MoM) estimator \citep{nemirovskij1983problem}.  
%or the self-normalized sum estimator \cite{minsker2021robust} are known to achieve such a guarantee with a probability of at least $1-\delta$. 
In this section, we assume a sample of size $n$ is provided to construct $\hat{\mu}$ and $\hat{\sigma}$. By avoiding sample splitting, we utilize the full sample for both estimators and maintain maximum efficiency. Next, we consider the estimator of $\sigma^2$ constructed as follows.
\begin{definition}
Let \( X_1, \dots, X_n \) be a sequence of i.i.d copies of a random variable \( X \) with mean \( \mu \) and finite variance \( \sigma^2 \). We partition the index set \([n]\) into \( K \) disjoint subsets \( (A_i^n)_{i \in [K]} \), where each subset \( A_i^n \) contains exactly \( m \) elements, i.e.,  
\[
(A_i^n)_{i \in [K]} \subset \left\{ J \subset [n] \mid |J| = m \right\}.
\]
The parameter \( m \) satisfies \( m \leq \lfloor n/K \rfloor \), and \( K \) is an integer such that \( K \leq \lfloor n/2 \rfloor \). 
For each subset \( A_i^n \), we define \( \hat{\sigma}_i^2 \) to be the usual empirical variance estimator of \( \sigma^2 \) for a given estimator \( \hat{\mu} \) of the mean \( \mu \), that is, 
\[
\forall i \in [K], \quad \hat{\sigma}_i^2 := \frac{1}{m} \sum_{j \in A_i^n} (X_j - \hat{\mu})^2.
\]
Finally, we define the median-of-means estimator via  
\[
(\hat{\sigma}_{\mathrm{MoM}}^{K})^2 := \operatorname{med}(\hat{\sigma}_1^2, \dots, \hat{\sigma}_K^2).
\]
\end{definition}

\begin{proposition}
\label{prop:sigma_estimator}
    Assume that $n\alpha_n = \Omega(\log(1/\delta))$ and that $P_\xi \in \mathcal{P}_{\alpha_n}(n,\delta)$. Let $\hat{\mu}$ be an estimator of $\mu$ satisfying $|\hat{\mu}- \mu| \leq 10 \sigma \sqrt{\frac{\log(1/\delta)}{n}}$ with probability at least $1-\delta$. Then the event 
    \[
    \mathbf{O}:= \left\{ \frac{\alpha_n}{4} \leq \frac{(\hat{\sigma}_{\mathrm{MoM}}^{K})^2 }{\sigma^2}\leq 10 \right\}
    \]
    occurs with probability at least $1-\delta/2$ whenever  $K=C\log(1/\delta)$ for an absolute constant $C>0$.
\end{proposition}
\begin{remark}
    Note that only the lower bound in the inequality depends on $\alpha_n$. It means that we can construct a \emph{lower} bound for $\sigma^2$ that is valid uniformly over $\m P_2$. However, the upper bound requires the lower bound for $\alpha_n$. 
\end{remark}
The estimator of the variance constructed above is actually minimax optimal in the sense that we cannot get better estimates of the variance within the class $\mathcal{P}_{\alpha_n}$. The next result proposes a matching lower bound.
\begin{proposition}
\label{prop:sigma_lower}
    Assume that $n\alpha_n = \Omega(\log(1/\delta))$. Then there exists an absolute constant $c>0$ such that no estimator of the variance $\tilde{\sigma}^2$ can satisfy the inequality
    \[
    \inf\limits_{\sigma>0,\mu\in\mb R, P_\xi\in \m P_{\alpha_n}(n,\delta)}\mathbb{P}\left( c \alpha_n \leq  \tilde{\sigma}^2/\sigma^2 \leq 1\right)\geq 1-\delta.
    \]
\end{proposition}

Combining both results provides minimax optimal confidence intervals for estimation of the variance $\sigma^2$. Notice that the optimal estimator $(\hat{\sigma}_{\mathrm{MoM}}^{K})^2$ can be employed beyond scenarios where $\alpha_n$ is a constant bounded away from $0$. 

\subsection{An optimal confidence interval}

For $\delta>0$ and $K=C\log(1/\delta)$, let $\hat{\mu}(\delta)$ and $\hat{\sigma}_{\mathrm{MoM}}^{K}$ be the estimators defined above. Consider the confidence interval 
\[
\hat{I}(\delta,\alpha_n) = \left[ \hat{\mu}(\delta) - 5\hat{\sigma}_{\mathrm{MoM}}^{K}\sqrt{\frac{\log(1/\delta)}{\alpha_n n}} , \hat{\mu}(\delta) + 5\hat{\sigma}_{\mathrm{MoM}}^{K}\sqrt{\frac{\log(1/\delta)}{\alpha_n    n}}\right].
\]
Then $\hat{I}(\delta,\alpha_n)$ leads to a minimax optimal test with respect to the risk $\mathcal{R}(n,\rho_n,\sigma_*,\alpha_n,\gamma_n,\delta)$: indeed, $\hat T$ defined via
\[
\hat{T}=\mathbf{1}\{0 \not\in \hat{I}(\delta)\},
\]
 achieves desired optimality. It is also clear that $|\hat{I}(\delta,\alpha_n)| \leq C\sigma\sqrt{\frac{\log(1/\delta)}{n\alpha_n}}$ with probability at least $1-\delta$. Unlike the problem of variance estimation, constructing $\hat{I}(\delta, \alpha_n)$ requires prior knowledge of the lower bound $\alpha_n$. On the one hand, our result remains adaptive to the choice of $\alpha_{n, \delta}(\xi)$ (conditionally on having access to its lower bound). On the other hand,  we showed that constructing adaptive confidence intervals is impossible without a specified lower bound on $\alpha_n$. This clarifies the cost of adaptation: while we do not need to know the exact distribution $P$ or the variance $\sigma^2$, the ``effective sample size'' $n\alpha_n$ dictates the fundamental limits of uncertainty quantification.

\section{Numerical experiments}

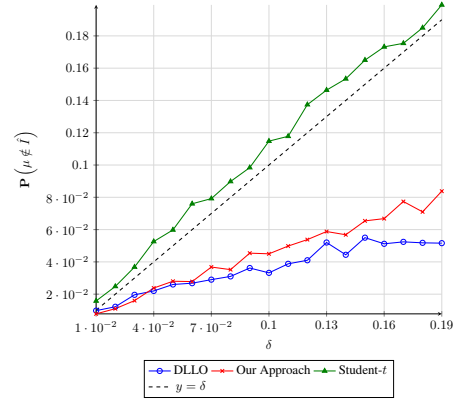
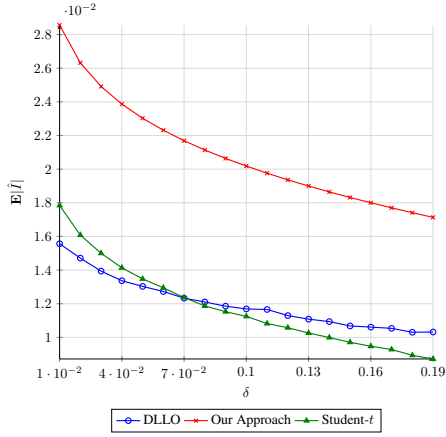
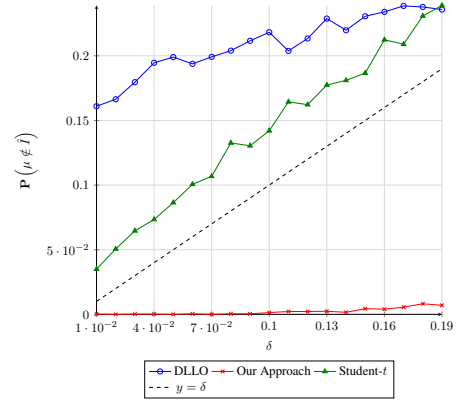
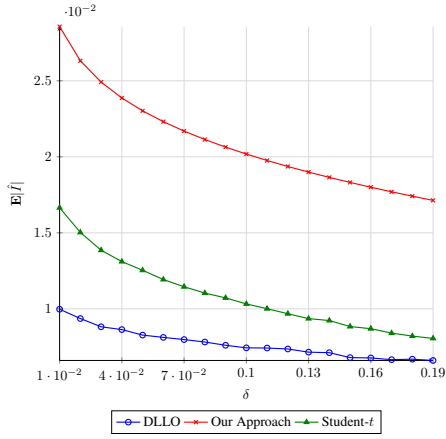
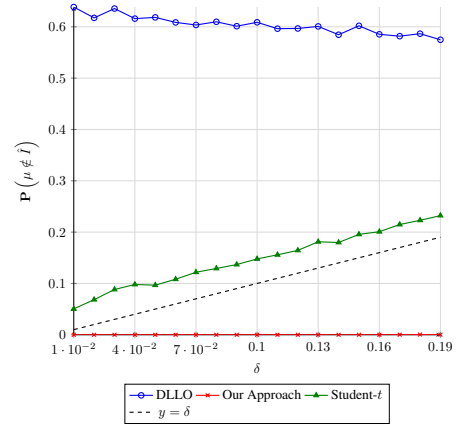
\begin{figure}[!ht]
    \centering

    \begin{subfigure}[b]{0.43\textwidth}
        \centering
        \resizebox{\textwidth}{!}{%
            \begin{tikzpicture}
    \begin{axis}[
        width=11cm, height=10cm,
        xlabel={$\delta$},
        ylabel={$\mathbf{E}|\hat{I}|$},
        xtick={0.01,0.04,...,0.19},
        grid=major,
        grid style={gray!30},
        axis lines=left,
        legend style={at={(0.5,-0.15)}, anchor=north, legend columns=3}
    ]
    \addplot[
        color=blue,
        thick,
        mark=o
    ]
    table[
        col sep=comma,
        x=delta,
        y=mean length
    ] {dllo_length_a1.csv};
    \addlegendentry{DLLO}
    \addplot[
        color=red,
        thick,
        mark=x
    ]
    table[
        col sep=comma,
        x=delta,
        y=mean length
    ] {ours_length_a1.csv};
    \addlegendentry{Our Approach}
    \addplot[
        color=green!50!black,
        thick,
        mark=triangle*
    ]
    table[
        col sep=comma,
        x=delta,
        y=mean length
    ] {student_length_a1.csv};
    \addlegendentry{Student-$t$}
    \end{axis}
\end{tikzpicture}
        }
        \caption{Mean length: $\alpha=1$.}
        \label{fig::mean_length_alpha_1}
    \end{subfigure}
    \hfill
    \begin{subfigure}[b]{0.43\textwidth}
        \centering
        \resizebox{\textwidth}{!}{%
            \begin{tikzpicture}
    \begin{axis}[
        width=11cm, height=10cm,
        xlabel={$\delta$},
        ylabel={$\mathbf{P}\left(\mu \notin \hat{I} \right)$},
        xtick={0.01,0.04,...,0.19},
        grid=major,
        grid style={gray!30},
        axis lines=left,
        legend style={at={(0.5,-0.15)}, anchor=north, legend columns=3}
    ]
    \addplot[
        color=blue,
        thick,
        mark=o
    ]
    table[
        col sep=comma,
        x=delta,
        y expr=1-\thisrow{mean coverage}
    ] {dllo_coverage_a1.csv};
    \addlegendentry{DLLO}
    \addplot[
        color=red,
        thick,
        mark=x
    ]
    table[
        col sep=comma,
        x=delta,
        y expr=1-\thisrow{mean coverage}
    ] {ours_coverage_a1.csv};
    \addlegendentry{Our Approach}
    \addplot[
        color=green!50!black,
        thick,
        mark=triangle*
    ]
    table[
        col sep=comma,
        x=delta,
        y expr=1-\thisrow{mean coverage}
    ] {student_coverage_a1.csv};
    \addlegendentry{Student-$t$}
    \addplot[
        black,
        dashed,
        thick,
        domain=0.01:0.19,
        samples=2
    ] {x};
    \addlegendentry{$y=\delta$}
    \end{axis}
\end{tikzpicture}
        }
        \caption{Miscoverage: $\alpha=1$.}
        \label{fig::mean_coverage_alpha_1}
    \end{subfigure}

    \vspace{1em}

    \begin{subfigure}[b]{0.43\textwidth}
        \centering
        \resizebox{\textwidth}{!}{%
            \begin{tikzpicture}
    \begin{axis}[
        width=11cm, height=10cm,
        xlabel={$\delta$},
        ylabel={$\mathbf{E}|\hat{I}|$},
        xtick={0.01,0.04,...,0.19},
        grid=major,
        grid style={gray!30},
        axis lines=left,
        legend style={at={(0.5,-0.15)}, anchor=north, legend columns=3}
    ]
    \addplot[
        color=blue,
        thick,
        mark=o
    ]
    table[
        col sep=comma,
        x=delta,
        y=mean length
    ] {dllo_length_a025.csv};
    \addlegendentry{DLLO}
    \addplot[
        color=red,
        thick,
        mark=x
    ]
    table[
        col sep=comma,
        x=delta,
        y=mean length
    ] {ours_length_a025.csv};
    \addlegendentry{Our Approach}
    \addplot[
        color=green!50!black,
        thick,
        mark=triangle*
    ]
    table[
        col sep=comma,
        x=delta,
        y=mean length
    ] {student_length_a025.csv};
    \addlegendentry{Student-$t$}
    \end{axis}
\end{tikzpicture}
        }
        \caption{Mean length: $\alpha=0.25$.}
        \label{fig:mean-length-alpha-025}
    \end{subfigure}
    \hfill
    \begin{subfigure}[b]{0.43\textwidth}
        \centering
        \resizebox{\textwidth}{!}{%
            \begin{tikzpicture}
    \begin{axis}[
        width=11cm, height=10cm,
        xlabel={$\delta$},
        ylabel={$\mathbf{P}\left(\mu \notin \hat{I} \right)$},
        xtick={0.01,0.04,...,0.19},
        grid=major,
        grid style={gray!30},
        axis lines=left,
        legend style={at={(0.5,-0.15)}, anchor=north, legend columns=3}
    ]
    \addplot[
        color=blue,
        thick,
        mark=o
    ]
    table[
        col sep=comma,
        x=delta,
        y expr=1-\thisrow{mean coverage}
    ] {dllo_coverage_a025.csv};
    \addlegendentry{DLLO}
    \addplot[
        color=red,
        thick,
        mark=x
    ]
    table[
        col sep=comma,
        x=delta,
        y expr=1-\thisrow{mean coverage}
    ] {ours_coverage_a025.csv};
    \addlegendentry{Our Approach}
    \addplot[
        color=green!50!black,
        thick,
        mark=triangle*
    ]
    table[
        col sep=comma,
        x=delta,
        y expr=1-\thisrow{mean coverage}
    ] {student_coverage_a025.csv};
    \addlegendentry{Student-$t$}
    \addplot[
        black,
        dashed,
        thick,
        domain=0.01:0.19,
        samples=2
    ] {x};
    \addlegendentry{$y=\delta$}
    \end{axis}
\end{tikzpicture}
        }
        \caption{Miscoverage: $\alpha=0.25$.}
        \label{fig:miscoverage-alpha-025}
    \end{subfigure}

    \vspace{1em}

    \begin{subfigure}[b]{0.43\textwidth}
        \centering
        \resizebox{\textwidth}{!}{%
            \begin{tikzpicture}
    \begin{axis}[
        width=11cm, height=10cm,
        xlabel={$\delta$},
        ylabel={$\mathbf{E}|\hat{I}|$},
        xtick={0.01,0.04,...,0.19},
        grid=major,
        grid style={gray!30},
        axis lines=left,
        legend style={at={(0.5,-0.15)}, anchor=north, legend columns=3}
    ]
    \addplot[
        color=blue,
        thick,
        mark=o
    ]
    table[
        col sep=comma,
        x=delta,
        y=mean length
    ] {dllo_length_a01.csv};
    \addlegendentry{DLLO}
    \addplot[
        color=red,
        thick,
        mark=x
    ]
    table[
        col sep=comma,
        x=delta,
        y=mean length
    ] {ours_length_a01.csv};
    \addlegendentry{Our Approach}
    \addplot[
        color=green!50!black,
        thick,
        mark=triangle*
    ]
    table[
        col sep=comma,
        x=delta,
        y=mean length
    ] {student_length_a01.csv};
    \addlegendentry{Student-$t$}
    \end{axis}
\end{tikzpicture}
        }
        \caption{Mean length: $\alpha=0.1$.}
        \label{fig:mean-length-alpha-01}
    \end{subfigure}
    \hfill
    \begin{subfigure}[b]{0.43\textwidth}
        \centering
        \resizebox{\textwidth}{!}{%
            \begin{tikzpicture}
    \begin{axis}[
        width=11cm, height=10cm,
        xlabel={$\delta$},
        ylabel={$\mathbf{P}\left(\mu \notin \hat{I} \right)$},
        xtick={0.01,0.04,...,0.19},
        grid=major,
        grid style={gray!30},
        axis lines=left,
        legend style={at={(0.5,-0.15)}, anchor=north, legend columns=3}
    ]
    \addplot[
        color=blue,
        thick,
        mark=o
    ]
    table[
        col sep=comma,
        x=delta,
        y expr=1-\thisrow{mean coverage}
    ] {dllo_coverage_a01.csv};
    \addlegendentry{DLLO}
    \addplot[
        color=red,
        thick,
        mark=x
    ]
    table[
        col sep=comma,
        x=delta,
        y expr=1-\thisrow{mean coverage}
    ] {ours_coverage_a01.csv};
    \addlegendentry{Our Approach}
    \addplot[
        color=green!50!black,
        thick,
        mark=triangle*
    ]
    table[
        col sep=comma,
        x=delta,
        y expr=1-\thisrow{mean coverage}
    ] {student_coverage_a01.csv};
    \addlegendentry{Student-$t$}
    \addplot[
        black,
        dashed,
        thick,
        domain=0.01:0.19,
        samples=2
    ] {x};
    \addlegendentry{$y=\delta$}
    \end{axis}
\end{tikzpicture}
        }
        \caption{Miscoverage: $\alpha=0.1$.}
        \label{fig:miscoverage-alpha-01}
    \end{subfigure}

    \caption{
        Comparison of mean interval length and empirical miscoverage for DLLO,
        our proposed adaptive approach, and the Student-$t$ interval under the
        considered mixture distributions. The dashed line in the right panels is the target
        miscoverage level $y=\delta$.
    }
    \label{fig:camera-ready-mixture}
\end{figure}

In this section, we provide a numerical illustration comparing our adaptive approach to the construction proposed by \citet{devroye2016sub} that we shall denote ``DLLO''. We focus on a scenario where the ``approximate symmetry'' condition \eqref{assumption:1} required by their method fails, as formally established in Lemma \ref{lem:counter}. To this end, we utilize the mixture distribution constructed in the proof of that lemma for different values of $\alpha$, which serves as a stress test for uncertainty quantification under heavy tails and asymmetry. For a given confidence level $1-\delta$, both procedures are calibrated by setting the number of blocks $K$ of order $\log(1/\delta)$ to ensure the nominal coverage is met in the well-behaved regime. In the experiments that follow, we have fixed $\alpha_n = 0.25$ for the construction of our adaptive procedure, representing a robust middle-ground for distributions that deviate from symmetry but retain sufficient ``effective'' mass for variance estimation.

\textbf{Calibration and baselines.}
All practical constants are calibrated once in the well-behaved regime
$\alpha=1$ and are subsequently kept fixed across all experiments.
Specifically, we use $K=\lceil 7\log(3/\delta)\rceil$ blocks and the
numerical constant $C=3.7$. The lower bound $\alpha_n$ supplied to our
procedure is also fixed once and for all at $\alpha_n=0.25$.
We then vary only the data-generating distribution, considering the
imbalanced mixtures from Lemma \ref{lem:counter} with
$\alpha\in\{1,0.25,0.1\}$.
We compare our method with the DLLO procedure of Devroye et al. and
with the Student-$t$ interval as an asymptotic benchmark.

The right panels report empirical miscoverage probabilities; hence the dashed line \(y=\delta\) represents the nominal boundary. In the well-behaved regime \(\alpha=1\), the robust procedures achieve nominal coverage. As \(\alpha\) decreases, DLLO substantially exceeds the nominal miscoverage level and the Student-\(t\) interval also exhibits undercoverage, whereas our procedure remains below the nominal boundary. The same calibration constants are used throughout, so the improved coverage in the difficult regimes is not the result of regime-specific tuning.

Figures \ref{fig::mean_length_alpha_1} and \ref{fig::mean_coverage_alpha_1} illustrate the regime where $\alpha \approx 1$. In this case, both the mean interval length and the empirical coverage of our method are comparable to those achieved by the approach of DLLO. This suggests that our refinement does not incur a significant ``robustness price'' when the distribution is well-behaved. However, as we move away from the easy regime, the performance gap becomes evident.
\begin{itemize}
\item \textbf{Moderate $\alpha$:} for $\alpha = 0.25$ (Figures \ref{fig:mean-length-alpha-025} and \ref{fig:miscoverage-alpha-025}), the method in \cite{devroye2016sub} begins to exhibit a degradation in reliability.
\item \textbf{Small $\alpha$:} for $\alpha = 0.1$ (Figures \ref{fig:mean-length-alpha-01} and \ref{fig:miscoverage-alpha-01}), the coverage of the standard approach drops significantly below the nominal level.
\end{itemize}
While our  intervals are naturally wider to account for the increased difficulty of the problem, they consistently maintain the prescribed $1-\delta$ coverage guarantee. This demonstrates that our method maintains nominal empirical coverage across the considered regimes, whereas non-adaptive methods fail to remain valid when the distribution's concentration properties are poor.
\begin{remark}
    Numerical experiments were performed using a standard laptop computer and did not require significant computational resources.
\end{remark}

 \subsection*{Final remarks and future directions} 

 We showed that the difficulty of the problem of constructing sub-Gaussian non-asymptotic confidence intervals for the mean is driven by the ``effective sample size''. This parameter offers natural interpolation between the ``easy'' and ``hard'' cases of the problem. Moreover, we established that it is not possible to learn this effective sample size from the data alone, and that a known lower bound is required for constructing valid confidence intervals. 
 % A careful inspection of the proofs reveals that our results still hold if instead of the finite variance one assumes existence of finite central moments of order $1+\delta$ for $\delta\in (0,1)$. The most significant change in this case is the fact that the length of confidence intervals will scale as $(n\alpha_n)^{-\frac{\delta}{1+\delta}}$; we limit the presentation to the case when $\delta = 1$ in order to streamline the presentation of main ideas.
 It would be interesting to extend our results to the case of distributions and confidence sets in $\mb R^d$, and find the correct extension of the notion of ``effective sample size'' to this more general framework.

\begin{ack}
The work of M.N. was supported by France $2030$. The work of S.M. was supported by the National Science Foundation grant DMS CAREER-2045068. The authors declare no competing interests.
\end{ack}
\bibliography{ref.bib}

\appendix

%###########################
\section{Proofs of the main results}
\label{section:proofs_upper}
%###########################

%#############################################

Everywhere below, $\mathbb{V}(X)$ denotes the variance of a random variable $X$.

\subsection{Proof of Proposition \ref{prop:infty}}

To show that the minimax risk $\mathcal{R}(n, \rho_n, \sigma_*) = 1$, it suffices to construct two distributions, $P_0$ and $P_1$, that are statistically indistinguishable based on a sample of size $n$, yet have different means. To this end, let $\tau \in (0, 1)$ be a small constant to be chosen later. Without loss of generality, we set $\sigma_\ast$ to be $1$ (the general case can be deduced via rescaling). Consider first the distribution $P_0$ (under $\mathbf{H}_0$) defined by
\[
\begin{cases}
    P_0(X = -A) = 1-\tau, \\
    P_0(X = A/\tau-A) = \tau.
\end{cases}
\]
For $X \sim P_0$, the mean is $\mathbb{E}[X] = 0$. The variance is given by $\sigma_0^2 = A^2(1-\tau)/\tau$, which is finite for any $A > 0$. Therefore, $P_0 \in \m P^{(2)}$. Now, consider the distribution $P_1$  defined via
\[
\begin{cases}
    P_1(X = -A) = 1- \tau, \\
    P_1(X = 1/\sqrt{\tau} - A) = \tau/2, \\
    P_1(X = -1/\sqrt{\tau} - A) = \tau/2.
\end{cases}
\]
It is easy to see that $\mathbb{E}[X] = -A$ and $\mathbb{V}(X) = 1$. 
Next, note that 
\[
\mb P_{X_1\sim P_0}\l(X_1 = \dots = X_n = -A\r) = \mb P_{X_1\sim P_1}\l(X_1 = \dots = X_n = -A\r) = (1-\tau)^n.
\]
Therefore, for any test $\hat T(X_1,\ldots,X_n)$, using the Neyman-Pearson lemma, we get
\begin{multline*}
\mb P_{X_1\sim P_0}\l(\hat T = 1 \r) + \mb P_{X_1\sim P_1}\l(\hat T = 0\r)
 \geq \int \min\left( dP_0^{\otimes n}, dP_1^{\otimes n} \right)\\
\geq 
\min\l(\mb P_{X_1\sim P_0}\l(X_1 = \dots = X_n = -A\r),\mb P_{X_1\sim P_1}\l(X_1 = \dots = X_n = -A\r)\r) = (1-\tau)^n.
\end{multline*}
Choosing $\tau$ to be small enough, we can make the right-hand side arbitrarily close to $1$.

Note that this simple construction demonstrates that without further restrictions on the distribution class, the presence of a dominant mass at the mean of a distribution makes adaptive uncertainty quantification impossible.
    \begin{comment}
        
\[
\begin{cases}
    P_0(X = 0) = 1- \tau, \\
    P_0(X = 1/\sqrt{\tau}) = \tau/2, \\
    P_0(X = -1/\sqrt{\tau}) = \tau/2.
\end{cases}
\]
For $X \sim P_0$, we have $\mathbb{E}[X] = 0$ and $\mathbb{V}(X) = 1$, thus $P_0 \in \m P^{(2)}$ with $\sigma = 1$. Now, consider the distribution $P_1$ (under $\mathbf{H}_1$) defined by
\[
\begin{cases}
    P_1(X = 0) = 1- \tau, \\
    P_1(X = A/\tau) = \tau.
\end{cases}
\]
For $X \sim P_1$, the mean is $\mathbb{E}[X] = A$. The variance is given by $\sigma_1^2 = A^2(1-\tau)/\tau$, which is finite for any $A > 0$. Therefore, $P_1 \in \m P^{(2)}$. The total variation distance between the product measures $P_0^{\otimes n}$ and $P_1^{\otimes n}$ is bounded by the probability of observing at least one non-zero value. Since the mass at the origin for both distributions is $1-\tau$, it follows that
\[
\int P_0^{\otimes n} \wedge P_1^{\otimes n} \geq \mathbf{P}_{P_0}(X_1 = \dots = X_n = 0) = (1-\tau)^n.
\]
By choosing $\tau \leq \frac{1}{n} \log(1/\delta)$, the right-hand side can be made arbitrarily close to $1$. In particular, for any $A > 0$, the distributions are indistinguishable as $n$ remains fixed and $\tau \to 0$. Since $\rho_n$ can be arbitrarily large and the Type I error is uniform over $\sigma$, the minimax risk $\mathcal{R}(n, \rho_n, \sigma_*)$ is necessarily $1$. This construction demonstrates that without further restrictions on the distribution class, the presence of a dominant mass at the mean makes adaptive uncertainty quantification impossible.

    \end{comment}

\subsection{Proofs of Proposition \ref{prop:alpha} and Theorem \ref{theorem:main}}
\label{proof:theorem:main}
\textbf{1. Proof of the lower bound.} Assume that $\frac{\log(1/\delta)}{n}\leq \alpha_n \leq \frac14$. To establish the lower bound, we will construct two distributions, $P_0$ and $P_1$, that are statistically indistinguishable under $n$ observations but have sufficiently different means. Let $A = \sqrt{\frac{\log(1/\delta)}{n\alpha_n}}$ and define $P_0$ via
\[
P_0 = \left(1 - \frac{\log(1/\delta)}{n} \right) \mathcal{N}(A, 1) + \frac{\log(1/\delta)}{n} \delta_{z_0}, \quad \text{where } z_0 = -\frac{1 - \frac{\log(1/\delta)}{n}}{\frac{\log(1/\delta)}{n}} A.
\]
Let $X \sim P_0$. By construction, $\mathbb{E}(X) = 0$. A direct calculation of the variance yields that 
\[
\mathbb{V}(X) = \left(1 - \frac{\log(1/\delta)}{n} \right)(A^2 + 1) + \frac{\log(1/\delta)}{n} z_0^2 \simeq \frac{1}{\alpha_n}.
\]
Let $\xi = X / \sqrt{\mathbb{V}(\xi)}$ be the standardized version. To show that $P_0 \in \mathcal{P}_{\alpha_n}$, we verify the effective mass condition. For $n\alpha_n \gg \log(1/\delta)$, we have that
\begin{align*}
\mathbb{E}\left({\xi}^2 \mathbf{1}\left\{{\xi}^2 \leq \frac{n\alpha_n}{\log(1/\delta)}\right\}\right) \geq \alpha_n\left(1 - \frac{\log(1/\delta)}{n} \right) \mathbb{E}\left[Z^2 \mathbf{1}\left\{Z^2 \leq \frac{n}{\log(1/\delta)}\right\}\right] \geq c\alpha_n,
\end{align*}
where $Z \sim \mathcal{N}(A, 1)$ and $c$ is a sufficiently small universal constant. Thus, $P_0 \in \mathcal{P}_{\alpha_n}$.

\noindent Next, consider the alternative distribution 
\[
P_1 = \left(1 - \frac{\log(1/\delta)}{n} \right) \mathcal{N}(A, 1) + \frac{\log(1/\delta)}{n} \delta_0.
\]
For $X \sim P_1$, the mean satisfies $\mathbb{E}(X) = c A = c \sqrt{\frac{\log(1/\delta)}{n\alpha_n}}$ for some $c>0$, and since $P_1$ has finite variance, $P_1 \in \m P^{(2)}$. We standardize $X$ by setting $\zeta = (X-\mathbb{E}(X))/\sqrt{\mathbb{V}(X)}$, and it is straightforward to verify that $\zeta$ satisfies the effective mass condition $\alpha_{n,\delta}(\zeta) \geq c_2$ for some constant $c_2>0$ that does not depend on $n$ and $\delta$. Since $\alpha_{n,\delta}(\zeta) \geq c_2$ then  $P_1 \in \mathcal{P}_{\gamma_n}$ for any $\gamma_n \leq c_2$. Using the Neyman-Pearson lemma, the minimax risk is lower-bounded by the total variation distance between the product measures :
\begin{multline*}
\mathcal{R}(n, \rho_n, \sigma_*, \alpha_n, \delta) \geq \int \min\left( dP_0^{\otimes n}, dP_1^{\otimes n} \right) \
\geq \mathbb{P}_{P_1}\left( \forall i \in [n], X_i \neq 0 \right) 
\\
= \left(1 - \frac{\log(1/\delta)}{n} \right)^n \geq e^{-\log(1/\delta)} = \delta.
\end{multline*}
This confirms that no test can distinguish $P_0$ and $P_1$ with risk better than $\delta$, establishing that the separation condition $\rho_n =\Omega\left( \sqrt{\frac{\log(1/\delta)}{n\alpha_n}}\right)$ is necessary.

\noindent \textbf{2. Proof of the upper bound.} Consider the test $\hat{T} := \mathbf{1}\left\{ \left| \frac{\hat{\mu}}{\hat{\sigma}} \right| \geq 5\sqrt{\frac{\log(1/\delta)}{n \alpha_n}} \right\}$. We recall that $|\hat{\mu} - \mu| \leq 10 \sigma \sqrt{\frac{\log(1/\delta)}{n}}$ with probability at least $1-\delta/2$, and that $\hat{\sigma}$ satisfies the bound of Proposition \ref{prop:sigma_estimator}.

\noindent\textbf{Type \Romannum{1} Error:} Under $\mathbf{H}_0$ ($\mu=0$), using the union bound and the concentration properties of $\hat{\sigma}$, we deduce that 
\begin{align*}
\mathbb{P}_0(\hat{T}=1) &\leq \mathbb{P}_0\left( \left|\frac{\hat{\mu}}{\sigma} \right| \geq 10\sqrt{\frac{\log(1/\delta)}{n}}\right) + \mathbb{P}_0 \left( \frac{\hat{\sigma}^2}{\sigma^2} \leq \frac{\alpha_n}{4} \right) \leq \frac{\delta}{2} + \frac{\delta}{2} = \delta.
\end{align*}

\noindent\textbf{Type II Error:} Under $\mathbf{H}_1$, assume that $|\mu| \geq 30\sigma \sqrt{\frac{\log(1/\delta)}{n\alpha_n}}$. By triangle inequality,
\begin{align*}
\mathbb{P}_1(\hat{T}=0) &= \mathbb{P}_1\left( \left|\hat{\mu}\right| < 5 \hat{\sigma} \sqrt{\frac{\log(1/\delta)}{n\alpha_n}} \right) \\
&\leq \mathbb{P}_1\left( |\mu| - |\hat{\mu} - \mu| < 5 \hat{\sigma} \sqrt{\frac{\log(1/\delta)}{n\alpha_n}} \right) \\
&\leq \mathbb{P}_1\left( |\hat{\mu} - \mu| > 30\sigma\sqrt{\frac{\log(1/\delta)}{n\alpha_n}} - 5 \hat{\sigma} \sqrt{\frac{\log(1/\delta)}{n\alpha_n}} \right).
\end{align*}
On the event $\l\{\hat{\sigma}^2 \leq 10\sigma^2\r\}$, we have that $5\hat{\sigma} \leq 5\sqrt{10}\sigma \leq 16 \sigma$. Therefore,
\begin{align*}
\mathbb{P}_1(\hat{T}=0) &\leq \mathbb{P}\left( |\hat{\mu} - \mu| > 14\sigma \sqrt{\frac{\log(1/\delta)}{n\alpha_n}} \right) + \mathbb{P}\left( \hat{\sigma}^2 > 10\sigma^2 \right) \leq \delta.
\end{align*}
The result follows. Observe that as $\alpha_n \to 0$, we recover the impossibility result of Proposition \ref{prop:infty}.

%A similar bound holds for Type \Romannum{2} error or for the length of the confidence interval for inference.

%###########################
\subsection{Proof of Proposition \ref{prop:sigma_estimator}}
\label{proof:sigma_estimator}
%###########################

Let $\mathcal{A}$ be the event $\{ | \hat{\mu}- \mu | \leq 10\sigma \sqrt{\frac{\log(1/\delta)}{n}}\}$, which holds with probability at least $1-\delta$ by assumption. On $\mathcal{A}$, we have that 
\[
(\hat{\mu}-\mu)^2 \leq 100 \sigma^2 \frac{\log(1/\delta)}{n}.
\]
We aim to bound the probability that the median-of-means estimator $\hat{\sigma}_{\mathrm{MoM}}^2$ significantly deviates from the true variance $\sigma^2$. By the definition of the median, this occurs only if more than half of the block estimators $\hat{\sigma}_k^2$ turn out to be ``bad". Specifically, we wish to show that
\[
\mathbb{P}\left( \frac{\alpha_n}{4} \leq \frac{\hat{\sigma}_{\mathrm{MoM}}^2}{\sigma^2} \leq 10 \right) \geq 1 - 2\delta.
\]
Recall that for any block $k \in [K]$, the estimator is $\hat{\sigma}_k^2 = \frac{1}{m} \sum_{j \in A_k} (X_j - \hat{\mu})^2$. Using the identity $(X_j - \hat{\mu})^2 \leq 2(X_j - \mu)^2 + 2(\hat{\mu} - \mu)^2$ (and similar lower bounds), we can relate the block estimators to the oracle centered sums $S_k := \frac{1}{m} \sum_{j \in A_k} \xi_j^2$, where $\xi_j = (X_j - \mu)/\sigma$. 
On event $\mathcal{A}$, the condition $\frac{\alpha_n}{4} \sigma^2 \leq \hat{\sigma}_k^2 \leq 10 \sigma^2$ is satisfied if
\[
\frac{\alpha_n}{2} + 200 \frac{\log(1/\delta)}{n} \leq S_k \leq 5 - 100 \frac{\log(1/\delta)}{n}.
\]
Under the assumption $n\alpha_n \geq C \log(1/\delta)$, we can simplify these thresholds to $S_k \in [\frac{2}{3}\alpha_n, 4]$. By the union bound, the probability that the median fails is bounded from above by
\begin{equation*}
\mathbb{P}(\mathcal{A}^c) + \mathbb{P}\left( \sum_{k=1}^K \mathbf{1}\{{S_k > 4}\} \geq \frac{K}{4} \right) + \mathbb{P}\left( \sum_{k=1}^K \mathbf{1}\{{S_k < \frac{2}{3}\alpha_n}\} \geq \frac{K}{4} \right).
\end{equation*}

\noindent\textbf{1. Upper bound (concentration of the mean).} Using Markov's inequality and the fact that $\mathbb{E}[S_k] = 1$, we see that $\mathbb{P}(S_k > 4) \leq 1/4$. However, to deduce a tighter bound, we assume $n$ is large enough such that $\mathbb{P}(S_k > 4) \leq 1/8$. Applying Hoeffding's inequality to the sum of indicators, we obtain that
\[
\mathbb{P}\left( \sum_{k=1}^K \mathbf{1}_{\{S_k > 4\}} \geq \frac{K}{4} \right) \leq \exp\left(- \frac{K}{150} \right) \leq \delta,
\]
provided that $K \geq 150 \log(1/\delta)$.

\noindent\textbf{2. Lower bound.} To control the lower tail, we consider the truncated variables $\bar{\xi}_j^2 = \xi_j^2 \mathbf{1}_{\{\xi_j^2 \leq n\alpha_n / \log(1/\delta)\}}$. By the definition of $\alpha_n$, we have $\mathbb{E}[\bar{\xi}_j^2] = \alpha_n$. We need to show that $\mathbb{P}(S_k < \frac{2}{3}\alpha_n) \leq 1/8$. Note that $\mathbb{V}(\bar{\xi}_j^2) \leq \mathbb{E}[\bar{\xi}_j^4] \leq \frac{n\alpha_n}{\log(1/\delta)} \mathbb{E}[\bar{\xi}_j^2]$. Applying Bernstein's inequality to the sum of $m$ i.i.d. truncated variables, we see that
\begin{align*}
\mathbb{P}\left( \frac{1}{m} \sum_{j=1}^m \bar{\xi}j^2 < \frac{2}{3}\alpha_n \right) &\leq \mathbb{P}\left( \frac{1}{m} \sum_{j=1}^m \bar{\xi}_j^2 - \mathbb{E}[\bar{\xi}_j^2] < - \frac{\alpha_n}{3} \right) \\
&\leq \exp \left( - \frac{m \alpha_n^2 / 18}{\mathbb{V}(\bar{\xi}_j^2) + \frac{n\alpha_n}{3 \log(1/\delta)} \frac{\alpha_n}{3}} \right) \\
&\leq \exp \left( - \frac{m \alpha_n^2}{C' \frac{n \alpha_n^2}{\log(1/\delta)}} \right) = \exp \left( - \frac{m \log(1/\delta)}{C' n} \right).
\end{align*}
Since $m = n/K$, this probability is bounded by a constant (e.g., $1/8$) as long as $K$ is of order $\log(1/\delta)$. Finally, applying Hoeffding's inequality again for the $K$ blocks ensures that the number of blocks failing the lower threshold is less than $K/4$ with probability $1-\delta$. Combining these results yields the claim.

\subsection{Proof of Proposition \ref{prop:sigma_lower}}
\label{proof:sigma_lower}
%###########################

Assume that $\frac{\log(1/\delta)}{n}\leq \alpha_n \leq \frac14$. We will establish the lower bound by constructing two triples $(\mu_0, \sigma_0, P_0)$ and $(\mu_1, \sigma_1, P_1)$ such that the resulting data distributions are nearly indistinguishable, yet their variances differ by a factor of $\alpha$.

\noindent \textbf{1. Construction of $P_0$.} Let $A := \sqrt{\frac{\log(1/\delta)}{n}}$ and define $P_0$ as the distribution of $\xi$ via the mixture $(1 - \tau) \mathcal{N}(A, \alpha_n) + \tau \delta_{z_0}$, where $\tau = \frac{\log(1/\delta)}{n}$ and $z_0$ is chosen such that $\mathbb{E}(\xi) = 0$. Simple algebra shows that $z_0 = -\frac{1 - \frac{\log(1/\delta)}{n}}{\frac{\log(1/\delta)}{n}}A$. A direct calculation yields that $\mathbb{V}(\xi) = (1-\tau)(1-\tau+A^2+\alpha_n)$. Hence, $\mathbb{V}(\xi)$ is of order $1$. To verify the effective mass condition, let $Z \sim \mathcal{N}(A, \alpha_n)$. Then
\[
\mathbb{E}[\xi^2 \mathbf{1}_{\{\xi^2 \leq \frac{n\alpha_n}{\log(1/\delta)}\}}] \geq (1-\tau) \mathbb{E}[Z^2 \mathbf{1}_{\{Z^2 \leq \frac{n\alpha_n}{\log(1/\delta)}\}}] \geq \alpha_n - C\frac{\log(1/\delta)}{n} \geq \frac{\alpha_n}{2},
\]
where we used the fact that $A^2 = \log(1/\delta)/n$. Thus, $\xi \in \mathcal{P}_{c_1\alpha_n}$ for some constant $c_1$.

\noindent \textbf{2. Construction of $P_1$.} Let $P_1$ be the mixture $\sim (1-\tau) \mathcal{N}(A, \alpha_n) + \tau \delta_0$. For this distribution, the mean is $\mathbb{E}(X) =(1-\tau) A$ and the variance $\mathbb{V}(X)$ is of order $ \alpha_n$. We standardize $X$ by setting $\zeta = (X-\mathbb{E}(X))/\sqrt{\mathbb{V}(X)}$, and it is straightforward to verify that $\zeta$ satisfies the effective mass condition $\alpha_{n,\delta}(\zeta) \geq c_2$ for some constant $c_2>0$.

\noindent \textbf{3. Indistinguishability argument.} We will now compare the product measures $Q_0 = P_0^{\otimes n}$ (with $\sigma^2=1$) and $Q_1 = P_1^{\otimes n}$ (with $\sigma^2$ of order $ \alpha$). Any estimator $\tilde{\sigma}^2$ must satisfy the coverage requirement for both. Let $c_3 = \min(c_1, c_2)$. The minimax error is bounded below by the testing error between $Q_0$ and $Q_1$, namely 
\begin{multline*}
\sup_{P \in \mathcal{P}_{P\alpha_n}} \mathbb{P}\left( \left(c_3\alpha_n \leq \frac{\tilde{\sigma}^2}{\sigma^2} \leq 1\right)^c \right) 
\\
\geq \frac{1}{2} \left[ \mathbb{P}_{Q_0}\left( \tilde{\sigma}^2 < c_3\alpha_n \text{ or } \tilde{\sigma}^2 > 1 \right) + \mathbb{P}_{Q_1}\left( \frac{\tilde{\sigma}^2}{\alpha_n} < c_3\alpha_n \text{ or } \frac{\tilde{\sigma}^2}{\alpha_n} > 1 \right) \right].
\end{multline*}
Note that the intervals $[c_3\alpha_n, 1]$ and $[c_3\alpha-n^2, \alpha_n]$ are disjoint for sufficiently small $\alpha_n$. Consequently, for any $\tilde{\sigma}^2$, the probability of error is at least the probability that the two distributions cannot be distinguished. Applying the same total variation argument as in Proposition \ref{prop:infty}, we see that
\begin{multline*}
\mathcal{E} \geq \frac{1}{2} \int \min(dQ_0, dQ_1) 
\geq \frac{1}{2} \mathbb{P}{Q_1}(X_1, \dots, X_n \neq 0) 
= \frac{1}{2} \left(1 - \frac{\log(1/\delta)}{n} \right)^n \geq \frac{\delta}{2}.
\end{multline*}
This implies that no estimator can consistently distinguish between a variance of order $1$ and a variance of order $\alpha_n$ across the class $\mathcal{P}_{\alpha_n}$, completing the proof.

\subsection{Extension to finite moment of order $1<p<2$}
\label{app:lower-moments}

Fix $\beta\in(0,1]$ and write $p=1+\beta$. Given a random variable $X$
with mean $\mu$ and $p$-th central moment
\[
    v_p:=\bigl(\E|X-\mu|^p\bigr)^{1/p}\in(0,\infty),
\]
define $\xi=\xi(X):=\frac{X-\mu}{v_p}$ and 
\[
\alpha^{(\beta)}_{n,\delta}(\xi)
:=
\sup\left\{
a\in[0,1]:
\E\left[
|\xi|^p
\mathbf{1}\left\{
|\xi|^p\le \frac{na}{\log(1/\delta)}
\right\}
\right]
\ge a^\beta
\right\}.
\]
For $\alpha\in(0,1]$, let $\mathcal P^{(\beta)}_\alpha(n,\delta)$
denote the class of distributions satisfying
$\alpha^{(\beta)}_{n,\delta}(\xi)\ge \alpha$.
For $\beta=1$, this definition coincides with the effective-mass
parameter used in the main text.

\paragraph{Generalized result.}
Up to constants depending only on $\beta$, the conclusions of
Theorem~ \ref{theorem:main} remain valid with $\sigma$ replaced by $v_p$ and with the
rate
\[
    v_p
    \left(
        \frac{\log(1/\delta)}{n\alpha}
    \right)^{\frac{\beta}{1+\beta}}.
\]
More precisely, if
$n\alpha\ge C_\beta\log(1/\delta)$, there exists a confidence interval
$\widehat C^{(\beta)}_n$ such that
\[
    \inf_{P\in\mathcal P^{(\beta)}_\alpha(n,\delta)}
    \mathbb{P}\!\left(\mu\in\widehat C^{(\beta)}_n\right)
    \ge 1-\delta
\]
and, with probability at least $1-\delta$,
\[
    \left|\widehat C^{(\beta)}_n\right|
    \le
    C_\beta v_p
    \left(
        \frac{\log(1/\delta)}{n\alpha}
    \right)^{\frac{\beta}{1+\beta}}.
\]
Conversely, this rate cannot be improved, up to constants depending
only on $\beta$. The same lower-bound construction also yields the
analogue of the non-adaptivity statement in Theorem \ref{theorem:main}.

\paragraph{Proof: upper bound.}
Set $L=\log(1/\delta)$ and let $K\asymp_\beta L$.
Partition the observations into $K$ blocks of size
$m\asymp n/K$. A standard median-of-means estimator $\widehat\mu$
satisfies
\[
    |\widehat\mu-\mu|
    \le
    C_\beta v_p\left(\frac{L}{n}\right)^{\beta/p}
\]
with probability at least $1-\delta/3$.
Indeed, if $\overline X_B$ denotes the empirical mean on a block,
the standard $L_p$ inequality for sums of independent centered random
variables, valid for $1<p\le2$, gives
\[
    \E|\overline X_B-\mu|^p
    \le
    C_\beta m^{1-p}\E|X-\mu|^p
    =
    C_\beta v_p^p m^{-\beta}.
\]
Markov's inequality gives a fixed-probability bound on each block,
and median aggregation amplifies it to probability $1-\delta/3$.

We next estimate the $p$-th central moment. For every block $B_k$, set
\[
    \widehat v_k^p
    :=
    \frac1m\sum_{j\in B_k}|X_j-\widehat\mu|^p,
    \qquad
    \widehat v_{\rm MoM}^p
    :=
    \operatorname{med}\bigl(\widehat v_1^p,\ldots,\widehat v_K^p\bigr).
\]
Let
\[
    S_k:=\frac1m\sum_{j\in B_k}|\xi_j|^p,
    \qquad
    e:=\frac{\widehat\mu-\mu}{v_p}.
\]
Since $p\in(1,2]$,
\begin{equation}\label{A.1}
    2^{1-p}S_k-|e|^p
    \le
    \frac{\widehat v_k^p}{v_p^p}
    \le
    2^{p-1}\bigl(S_k+|e|^p\bigr).
    \tag{A.1}
\end{equation}
The upper tail of $S_k$ is controlled directly by Markov's inequality,
because $\E|\xi|^p=1$.

For the lower tail, take a witness
$a_\star\ge \alpha/2$ in the definition of
$\alpha^{(\beta)}_{n,\delta}(\xi)$ and define
\[
    Y
    :=
    |\xi|^p
    \mathbf{1}\left\{
        |\xi|^p\le\frac{na_\star}{L}
    \right\}.
\]
Then
\[
    q:=\E Y\ge a_\star^\beta,
    \qquad
    0\le Y\le \frac{na_\star}{L},
\]
and hence
\[
    \mathbb{V}(Y)
    \le
    \E Y^2
    \le
    \frac{na_\star}{L}\,\E Y.
\]
Bernstein's inequality therefore gives
\[
\begin{aligned}
    \mathbb{P}\left(
        \frac1m\sum_{j\in B_k}Y_j < \frac q2
    \right)
    &\le
    \exp\left(
        -c_\beta
        \frac{m q}{na_\star/L}
    \right)                                                     \\
    &\le
    \exp\left(
        -c_\beta \frac{mL}{n}a_\star^{\beta-1}
    \right)
    \le
    \exp\left(
        -c_\beta \frac{mL}{n}
    \right),
\end{aligned}
\tag{A.2}
\]
where the last inequality uses $\beta\le1$ and $a_\star\le1$.
Choosing the constants in $K\asymp_\beta L$ so that the blockwise
success probability is bounded away from $1/2$, and applying the
usual median amplification argument, yields
\begin{equation}\label{A.3}
        c_\beta\alpha^\beta
    \le
    \frac{\widehat v_{\rm MoM}^p}{v_p^p}
    \le
    C_\beta
    \tag{A.3}
\end{equation}
with probability at least $1-\delta/3$.
To use \eqref{A.1}, note that on the event controlling
$\widehat\mu$,
\[
    |e|^p
    \le
    C_\beta\left(\frac{L}{n}\right)^\beta
    \le
    c_\beta\alpha^\beta
\]
whenever $n\alpha\ge C_\beta L$.

Combining the mean and scale bounds, define
\[
    \widehat C^{(\beta)}_n
    :=
    \left[
    \widehat\mu
    \pm
    C_\beta\widehat v_{\rm MoM}
    \left(
        \frac{L}{n\alpha}
    \right)^{\beta/p}
    \right].
\]
The lower bound in \eqref{A.3} implies
$v_p\le C_\beta\alpha^{-\beta/p}\widehat v_{\rm MoM}$,
so the preceding deviation inequality for $\widehat\mu$ gives
coverage at least $1-\delta$. The upper bound in \eqref{A.3} gives
the stated bound on the interval length.

\paragraph{Proof: matching lower bound.}
We adapt the mixture construction used in the proof of
Theorem \ref{theorem:main}. Let
\[
    \tau:=\frac{L}{2n},
    \qquad
    A:=\left(\frac{\tau}{\alpha}\right)^{\beta/p},
\]
and let $\varepsilon$ be a Rademacher random variable. Consider the
common dominant component
\[
    Q:=\mathcal L(A+\varepsilon),
\]
where $\mathcal{L}$ denotes the law of a random variable. This construction is analogous to the common Gaussian component used in the proof of Theorem~\ref{theorem:main}.
Define
\[
    P_0=(1-\tau)Q+\tau\delta_{z_0},
    \qquad
    z_0:=-\frac{1-\tau}{\tau}A,
\]
and
\[
    P_1=(1-\tau)Q+\tau\delta_0.
\]
Then $\E_{P_0}X=0$, whereas
\[
    \mu_1:=\E_{P_1}X=(1-\tau)A\asymp A.
\]
Let $v_{0,p}$ and $v_{1,p}$ denote the corresponding $p$-th central
scales. Since $\tau\le\alpha$ under
$n\alpha\ge C_\beta L$, we have $A\le1$, and a direct calculation
gives
\begin{equation}\label{A.4}
    v_{0,p}^p\asymp_\beta
    \tau |z_0|^p
    \asymp_\beta
    A^p\tau^{1-p}
    =
    \alpha^{-\beta},
    \qquad
    v_{1,p}^p\asymp_\beta 1.
    \tag{A.4}
\end{equation}
After standardization, the dominant component of $P_0$ contributes
a $p$-th moment of order $\alpha^\beta$. Taking
$a_0=c_\beta\alpha$ in the definition of
$\alpha^{(\beta)}_{n,\delta}$, and choosing $c_\beta>0$ small enough,
the truncation threshold $na_0/L$ contains the whole dominant
component, while
\[
    \E_{P_0}\left[
        |\xi_0|^p
        \mathbf{1}\left\{
            |\xi_0|^p\le \frac{na_0}{L}
        \right\}
    \right]
    \ge
    c_\beta\alpha^\beta
    \ge
    a_0^\beta.
\]
Hence
$\alpha^{(\beta)}_{n,\delta}(\xi_0)\ge c_\beta\alpha$.
For $P_1$, the scale is of constant order and the centered dominant
component is bounded; the same argument with a fixed sufficiently
small $a_1>0$ gives
\[
    \alpha^{(\beta)}_{n,\delta}(\xi_1)\ge a_1,
\]
uniformly in $n,\delta$, and $\alpha$. Thus the null belongs to the
hard class of order $\alpha$, whereas the alternative belongs to an
easier class bounded away from zero.

Finally, $P_0$ and $P_1$ share the sub-probability measure
$(1-\tau)Q$. Therefore,
\[
    \int \min\{dP_0^{\otimes n},dP_1^{\otimes n}\}
    \ge
    (1-\tau)^n.
\]
Since $\tau=L/(2n)$ and $L\le n$,
$\log(1-x)\ge-2x$ for $0\le x\le1/2$ yields
\[
    (1-\tau)^n\ge e^{-L}=\delta.
\]
By the Neyman--Pearson lemma, every test between the two experiments
has total error at least $\delta$. On the other hand, by
\eqref{A.4},
\[
    \frac{|\mu_1-\mu_0|}{v_{1,p}}
    \asymp_\beta
    A
    =
    \left(
        \frac{L}{n\alpha}
    \right)^{\beta/(1+\beta)}.
\]
Rescaling the construction gives any prescribed value of $v_{1,p}$.
Consequently, no uniformly valid confidence interval can have
length of smaller order than
\[
    v_p
    \left(
        \frac{\log(1/\delta)}{n\alpha}
    \right)^{\beta/(1+\beta)}.
\]
Since the alternative has effective mass bounded away from zero, the
same pair also gives the analogue of the non-adaptivity conclusion of
Theorem \ref{theorem:main}. This completes the proof.

%###########################
\section{Auxiliary results}
\label{section:auxiliary}
%###########################

It is well known that when $X$ is Gaussian we can achieve the optimal deviations bounds for the confidence interval around the empirical mean $\Bar{X}_n$ such that $\forall \delta>0$, we have
\[
\mathbf{P}\left( |\Bar{X}_n - \mathbf{E}(X)| \geq s_n \sqrt{\frac{2(1+o(1))\log(1/\delta)}{n}}\right) = \delta,
\]
where $s_n$ denotes the empirical standard deviation. 
\begin{lemma}
\label{lemma:gaussian}
    Let $X_1,\ldots,X_n$ be i.i.d. with normal distribution $\mathcal{N}(\mu,\sigma^2)$. Then for all $e^{-n/2} \leq \delta\leq e^{-1}$,
    \[
    \mathbb{P}\left( |\Bar{X}_n - \mathbf{E}(X)| \geq s_n \sqrt{\frac{6\log(2/\delta)}{n}}\right) = \delta
    \]
Moreover, the bound holds with $\sqrt{6\log(2/\delta)}$ replaced by $\sqrt{2(1+o(1)\log(1/\delta)}$ where $o(1)\to 0$ as $n\to\infty$ for each $0<\delta\leq 1$.
\end{lemma}
\begin{proof}
    The probability density function of the Student's t-distribution with $n-1$ degrees of freedom is 
    \[
    p_n(x) = c_n \l(1 + \frac{x^2}{n-1}\r)^{-n/2}
    \]
    where $c_n = \frac{\Gamma\l( \frac{n}{2}\r) }{\sqrt{\pi (n-1)} \Gamma\l( \frac{n-1}{2}\r)}$. Therefore, letting $Z$ have t-distribution with $n-1$ d.f., we deduce that 
    \be{
    \pr{Z\geq t} = \int_{t}^\infty p_n(x) dx 
    \leq \int_{t}^\infty \frac{x}{t} c_n \l(1 + \frac{x^2}{n-1}\r)^{-n/2} 
    \leq \frac{c_n}{t}\frac{n-1}{n-2}\l( 1 + \frac{t^2}{n-1} \r)^{-(n-2)/2},
    }
    where we used the fact that $\frac{x}{t}\geq 1$ in the region of integration. It remains to simplify the resulting expression using the inequality $\frac{\Gamma\l( \frac{n}{2}\r) }{\Gamma\l( \frac{n-1}{2}\r)} \leq \sqrt{n/2}$ (Gautschi's inequality) so that \[
    c_n \frac{n-1}{n-2} \leq \frac{1}{\sqrt{2\pi}}\sqrt{\frac{n}{n-1}}\frac{n-1}{n-2} \leq \sqrt{\frac{16}{15\pi}}
    \]
    for $n\geq 5$. Moreover, since $(1+z)^{-p}\leq e^{-pz/2}$ for all $p>0$ and $0\leq z\leq \frac12$,  
    \[
    \l( 1 + \frac{t^2}{n-1} \r)^{-(n-2)/2} \leq e^{-t^2/6}
    \]
    for $n\geq 6$. The same calculation yields a bound for the lower tail, hence the result follows by setting $t=\sqrt{6\log(2/\delta)}$. 

    The second part of the claim stems from the well-known fact that the Student's t-statistic is asymptotically normally distributed \citep{gine1997student}.
\end{proof}

\begin{lemma}\label{lem:counter}
    For $\alpha$ small enough, there exists a distribution $\xi \in \mathcal{P}_\alpha$ such that
    $
    \min(p_-(j),p_+(j)) \leq 1/4$. \footnote{It is clear from the proof that $\frac14$ can be replaced by any positive constant.}
\end{lemma}

\begin{proof}

Assume $\log(1/\delta)/n \leq \alpha \leq 1/4$. We aim to show that for sufficiently small $\alpha$, there exists a distribution in $\mathcal{P}_\alpha$ that is highly asymmetric, such that the sample mean remains on one side of the true mean with high probability. Consider the distribution $P_0$ defined as in the proof of Proposition \ref{prop:sigma_lower}:$$P_0 = (1 - \tau) \mathcal{N}(A, \alpha) + \tau \delta_{z_0},$$where $\tau = \frac{\log(1/\delta)}{n}$, $A = \sqrt{\frac{\log(1/\delta)}{n}}$, and $z_0 = -\frac{1-\tau}{\tau}A$. As established previously, $P_0 \in \mathcal{P}_\alpha$ and $\mathbb{E}(\xi) = 0$.

Let $\xi_1, \dots, \xi_m$ be i.i.d. copies of $\xi \sim P_0$, where $m = \lfloor \frac{n}{10\log(1/\delta)} \rfloor$. We can represent each $\xi_i$ using an auxiliary Bernoulli variable $B_i \sim \text{Bern}(1-\tau)$:$$\xi_i = B_i Z_i + (1-B_i) z_0, \quad \text{where } Z_i \sim \mathcal{N}(A, \alpha).$$Now, consider the event $\mathcal{E} = \{B_1 = \dots = B_m = 1\}$, which corresponds to the case where no ``outlier" from the $\delta_{z_0}$ component is sampled. The probability of this event is:$$\mathbb{P}(\mathcal{E}) = (1-\tau)^m = \left(1 - \frac{\log(1/\delta)}{n}\right)^{\frac{n}{10\log(1/\delta)}} \geq \exp(-1/10) \approx 0.904.$$On the event $\mathcal{E}$, the sample mean $\bar{\xi}_m$ is simply the empirical mean of the Gaussian components $Z_i$. Thus:
\begin{align*}p_+(m) = \mathbb{P}\left( \frac{1}{m} \sum_{i=1}^m \xi_i \geq 0 \right) &\geq \mathbb{P}\left( \frac{1}{m} \sum_{i=1}^m Z_i \geq 0 \ \middle| \ \mathcal{E} \right) \mathbb{P}(\mathcal{E}) \\
&\geq \mathbb{P}\left( A + \sqrt{\frac{\alpha}{m}} \mathcal{N}(0,1) \geq 0 \right) e^{-1/10}.
\end{align*}
Since $A = \sqrt{\log(1/\delta)/n}$ and $m \approx n/(10\log(1/\delta))$, the term inside the probability is:$$\mathbb{P}\left( \mathcal{N}(0,1) \geq - A\sqrt{m/\alpha} \right) \approx \mathbb{P}\left( \mathcal{N}(0,1) \geq - \sqrt{\frac{1}{10\alpha}} \right).$$As $\alpha \to 0$, this probability tends to $1$. Consequently, $p_+(m) \geq e^{-1/10}(1 - \epsilon) > 3/4$ for small $\alpha$. Since $p_-(m) = 1 - p_+(m)$, it follows that:$$\min(p_-(m), p_+(m)) \leq 1 - e^{-1/10}(1-\epsilon) < 1/4.$$This demonstrates that the approximate symmetry condition \eqref{assumption:1} is violated for $P \in \mathcal{P}_\alpha$, proving that the construction in \cite{devroye2016sub} cannot be applied uniformly to this class.

\end{proof}

%%%%%%%%%%%%%%%%%%%%%%%%%%%%%%%%%%%%%%%%%%%%%%%%%%%%%%%%%%%%

\newpage
\section*{NeurIPS Paper Checklist}

\begin{enumerate}

\item {\bf Claims}
    \item[] Question: Do the main claims made in the abstract and introduction accurately reflect the paper's contributions and scope?
    \item[] Answer: \answerYes{} % Replace by \answerYes{}, \answerNo{}, or \answerNA{}.
    \item[] Justification: the abstract explains the main contributions while sections 1 and 2 make the claims mathematically precise.

\item {\bf Limitations}
    \item[] Question: Does the paper discuss the limitations of the work performed by the authors?
    \item[] Answer: \answerYes{} % Replace by \answerYes{}, \answerNo{}, or \answerNA{}.
    \item[] Justification: yes, we discuss limitations and possible extensions of our results in section \textbf{``Final remarks and future directions''} at the end of the manuscript.

\item {\bf Theory assumptions and proofs}
    \item[] Question: For each theoretical result, does the paper provide the full set of assumptions and a complete (and correct) proof?
    \item[] Answer: \answerYes{} % Replace by \answerYes{}, \answerNo{}, or \answerNA{}.
    \item[] Justification: while high-level intuition behind the proofs is outlined in the main text, the complete technical arguments are given in the supplementary material.

    \item {\bf Experimental result reproducibility}
    \item[] Question: Does the paper fully disclose all the information needed to reproduce the main experimental results of the paper to the extent that it affects the main claims and/or conclusions of the paper (regardless of whether the code and data are provided or not)?
    \item[] Answer: \answerYes{} % Replace by \answerYes{}, \answerNo{}, or \answerNA{}.
    \item[] Justification: detailed overview of numerical experiments is provided in section 5 of the paper.

\item {\bf Open access to data and code}
    \item[] Question: Does the paper provide open access to the data and code, with sufficient instructions to faithfully reproduce the main experimental results, as described in supplemental material?
    \item[] Answer: \answerYes{} % Replace by \answerYes{}, \answerNo{}, or \answerNA{}.
    \item[] Justification: Zip file with anonymized code has been submitted with the supplementary materials.

\item {\bf Experimental setting/details}
    \item[] Question: Does the paper specify all the training and test details (e.g., data splits, hyperparameters, how they were chosen, type of optimizer) necessary to understand the results?
    \item[] Answer: \answerYes{} % Replace by \answerYes{}, \answerNo{}, or \answerNA{}.
    \item[] Justification: details are given in section 5 of the paper.

\item {\bf Experiment statistical significance}
    \item[] Question: Does the paper report error bars suitably and correctly defined or other appropriate information about the statistical significance of the experiments?
    \item[] Answer: \answerNA{} % Replace by \answerYes{}, \answerNo{}, or \answerNA{}.
    \item[] Justification: 

\item {\bf Experiments compute resources}
    \item[] Question: For each experiment, does the paper provide sufficient information on the computer resources (type of compute workers, memory, time of execution) needed to reproduce the experiments?
    \item[] Answer: \answerYes{} % Replace by \answerYes{}, \answerNo{}, or \answerNA{}.
    \item[] Justification: numerical experiments were produced using a standard laptop computer and do not require significant computational resources.
    
\item {\bf Code of ethics}
    \item[] Question: Does the research conducted in the paper conform, in every respect, with the NeurIPS Code of Ethics \url{https://neurips.cc/public/EthicsGuidelines}?
    \item[] Answer: \answerYes{} % Replace by \answerYes{}, \answerNo{}, or \answerNA{}.
    \item[] Justification: The research is theoretical in nature and focuses on mathematical statistics. It does not involve human subjects, sensitive data, or potential for immediate harmful applications.

\item {\bf Broader impacts}
    \item[] Question: Does the paper discuss both potential positive societal impacts and negative societal impacts of the work performed?
    \item[] Answer: \answerNo{} % Replace by \answerYes{}, \answerNo{}, or \answerNA{}.
    \item[] Justification: this is primarily a theoretical paper and it does not have any immediate societal impacts.
    
\item {\bf Safeguards}
    \item[] Question: Does the paper describe safeguards that have been put in place for responsible release of data or models that have a high risk for misuse (e.g., pre-trained language models, image generators, or scraped datasets)?
    \item[] Answer: \answerNA{} % Replace by \answerYes{}, \answerNo{}, or \answerNA{}.
    \item[] Justification: 

\item {\bf Licenses for existing assets}
    \item[] Question: Are the creators or original owners of assets (e.g., code, data, models), used in the paper, properly credited and are the license and terms of use explicitly mentioned and properly respected?
    \item[] Answer: \answerNA{} % Replace by \answerYes{}, \answerNo{}, or \answerNA{}.
    \item[] Justification: 

\item {\bf New assets}
    \item[] Question: Are new assets introduced in the paper well documented and is the documentation provided alongside the assets?
    \item[] Answer: \answerNA{} % Replace by \answerYes{}, \answerNo{}, or \answerNA{}.
    \item[] Justification: 

\item {\bf Crowdsourcing and research with human subjects}
    \item[] Question: For crowdsourcing experiments and research with human subjects, does the paper include the full text of instructions given to participants and screenshots, if applicable, as well as details about compensation (if any)? 
    \item[] Answer: \answerNA{} % Replace by \answerYes{}, \answerNo{}, or \answerNA{}.
    \item[] Justification: 

\item {\bf Institutional review board (IRB) approvals or equivalent for research with human subjects}
    \item[] Question: Does the paper describe potential risks incurred by study participants, whether such risks were disclosed to the subjects, and whether Institutional Review Board (IRB) approvals (or an equivalent approval/review based on the requirements of your country or institution) were obtained?
    \item[] Answer: \answerNA{} % Replace by \answerYes{}, \answerNo{}, or \answerNA{}.
    \item[] Justification: 

\item {\bf Declaration of LLM usage}
    \item[] Question: Does the paper describe the usage of LLMs if it is an important, original, or non-standard component of the core methods in this research? Note that if the LLM is used only for writing, editing, or formatting purposes and does \emph{not} impact the core methodology, scientific rigor, or originality of the research, declaration is not required.
    %this research? 
    \item[] Answer: \answerNA{} % Replace by \answerYes{}, \answerNo{}, or \answerNA{}.
    \item[] Justification: 

\end{enumerate}

\end{document}